\documentclass[12pt]{article}

\usepackage{amsfonts,amsmath,amssymb,latexsym,xcolor,mathrsfs,tikz,breqn,extarrows}
\usepackage{amsthm,authblk}
\usepackage{multirow}
\usepackage{mathtools}
\usepackage{calc}
\usepackage{pgfplots}
\pgfplotsset{compat=1.18}
\usetikzlibrary{patterns,arrows,decorations.pathreplacing}
\usepackage{diagbox}
\usepackage{makecell}
\usepackage{float}

\def\q{\hfill\rule{1ex}{1ex}}
\def\0{\emptyset}

\newtheorem{theorem}{Theorem}[section]
\newtheorem{definition}[theorem]{Definition}
\newtheorem{lemma}[theorem]{Lemma}
\newtheorem{claim}[theorem]{Claim}

\newtheorem{cor}[theorem]{Corollary}
\newtheorem{prop}[theorem]{Proposition}

\newtheorem{prob}[theorem]{Problem}

\newcommand{\ind}{\boldsymbol{\cdot}}

\numberwithin{equation}{section}

\usepackage{algorithm}
\usepackage{algpseudocode}
\usepackage{graphicx}
\usepackage{pstricks,pgf,subcaption,amssymb}

\begin{document}
\title{\bf
Degree-restricted semi-saturation
numbers of cliques and its applications
	 }
\author[1,2]{
Zhen He}
\author[3]{Mei Lu}
\author[3]{
Yanzhe Qiu}
\author[3]{
Yiduo Xu\thanks{Corresponding author. E-mail:\texttt{xyd23@mails.tsinghua.edu.cn}}\;}

\affil[1]{\small School of Mathematics and Statistics, Beijing Jiaotong University, Beijing 100044, China.}
\affil[2]{\small Beijing Key Laboratory of Biological Big Data and Topological Statistics, Beijing Jiaotong University, Beijing, 100044, P.R. China.}
\affil[3]{\small Department of Mathematical Sciences, Tsinghua University, Beijing 100084, China}

\date{}

\maketitle\baselineskip 16.3pt

\begin{abstract}
	A graph $G$ is said to be $F$-semi-saturated if the addition of any nonedge $e \not \in E(G)$ would create a new copy of $F$ in $G+e$.
	The semi-saturation number $ssat(n,F)$ is the minimum number of edges in an $F$-semi-saturated graph of order $n$.
	In this paper we investigate the semi-saturation number of $K_r$ on $n$ vertices with maximal degree at most $\Delta$, denoted by $ssat^{\Delta}(n,K_r)$.
	This investigation was suggested by Erd\H os, R\'enyi and  S\'os, who in 1966 considered the graph of diameter 2  with degree restrictions, equivalently $ssat^{\Delta}(n,K_3)$.

	The following are some of our results. For arbitrary $r \geq 4$, we show that the limit $ \lim_{n \rightarrow \infty} ssat^{cn}(n,K_r)/n$ exists for all $0 < c \leq 1$, except for some sparse values of $c$ contained in a countable and rational sequence $c_i \rightarrow 0$.
	Moreover, we establish the asymptotic behaviour of this limit for $\frac{r}{r+2} < c <1$ and determine the exact value of $ssat^{\Delta}(n,K_r)$ for some specific $\Delta$.
	As an application, we determine the relation between the saturation number of the join graph $K_r \vee F$ and that of $F$ for a large class of pairs $(r,F)$.

	\end {abstract}


\section{Introduction}

In this paper we only consider finite, simple and undirected graphs. For a graph $G$, we use $V(G)$ to denote the vertex set of $G$, $E(G)$ the edge set of $G$, $|G|$ the order of $G$ and $e(G)$ the size of $G$. Let $u\in V(G)$ and $S\subseteq V(G)$. Denote $N(u)=\{v\in V(G):uv\in E(G)\}$,
 $N[u]=N(u)\cup\{u\}$, $N(S)=(\cup_{u\in S}N(u))\setminus S$ and $N[S]=N(S)\cup S$. The degree of $u$ is $d(u)=|N(u)|$. The minimum and maximum  degree of $G$ are  $\delta(G)=\min\{d_G(u)|u\in V(G)\}$ and $\Delta(G)=\max\{d_G(u)|u\in V(G)\}$, respectively.
Let $x,y\in V(G)$. The distance of $x$ and $y$ in $G$, denoted by $\mathrm{dist}_{G}(x,y)$, is the number of edges in the shortest path connecting $x$ and $y$.
For graphs $G_1,G_2$, let  $G_1 \cup G_2$ be the union of vertex-disjoint copy of $G_1,G_2$, and let $G_1 \vee G_2$ be the join of $G_1$ and $G_2$, obtained from $G_1 \cup G_2$ by adding all edges between $G_1$ and $G_2$.
	Given graphs $G$ and $F$, we say $G$ is \textit{$F$-free} if $G$ does not contain any copy of $F$. We say $G$ is \textit{$F$-semi-saturated} if the addition of any nonedge $e \not \in E(G)$ would create a new copy of $F$ in $G+e$. We say $G$ is \textit{$F$-saturated} if $G$ is $F$-free and $F$-semi-saturated. The \textit{saturation number} and \textit{semi-saturation number} of $F$ are denoted by
\begin{equation*}
\begin{aligned}
	   sat(n,F) & =\min \{ e(G):G  \; \text{is}  \; F\text{-saturated  and } |G|=n  \} \, , \\
	   ssat(n,F) & =\min \{ e(G):G  \; \text{is}  \; F\text{-semi-saturated  and } |G|=n  \} \, .
\end{aligned}
\end{equation*}

The first (semi)-saturation problem was studied by Erd\H os, Hajnal and Moon \cite{EHM} in 1964, determining that $ssat(n,K_r)=sat(n,K_r)=(r-2)n-\frac{(r-2)(r-3)}{2}$. For readers interested in the saturation problem, we refer to the survey \cite{survey}.

	 The (semi)-saturation number of $F$ with restriction on maximum degree is denoted by
\begin{equation*}
\begin{aligned}
	   (s)sat^{\Delta}_*(n,F) & =\min \{ e(G):G  \; \text{is}  \; F\text{-(semi)-saturated, } |G|=n \text{ and } \Delta(G)=\Delta \} \, ; \\
	   (s)sat^{\Delta}(n,F) & =\min \{ e(G):G  \; \text{is}  \; F\text{-(semi)-saturated, } |G|=n \text{ and } \Delta(G) \leq \Delta \} \, .
\end{aligned}
\end{equation*}
	If such an $F$-(semi)-saturated graph with restriction on maximum degree does not exist, we define $(s)sat^{ \Delta}_{(*)}(n,F) =  \infty$. For example, $sat^{r-2} (n, K_r)= \infty$ for $ n \geq r \geq 3$, since any graph that has maximum degree at most $r-2$ cannot be $K_r$-saturated.
	The above definitions clearly imply
\begin{equation}\label{eq:1.1}
\begin{aligned}
	   sat^{ \Delta}_{(*)}(n,F) & \geq ssat^{\Delta}_{(*)}(n,F) \; ; \\
	   (s)sat^{ \Delta}_{*}(n,F) & \geq (s)sat^{\Delta}(n,F) \; .
\end{aligned}
\end{equation}
For fixed $n$ and $F$, it is clear that
\begin{equation}\label{eq:1.2}
(s)sat(n,F)= \min \{ \, (s)sat^{ n-2}(n,F) \, , \; (s)sat^{n-1}_*(n,F) \, \} \, .
\end{equation}

The (semi)-saturation problem with restriction on maximum degree was first studied by Hajnal \cite{Hajnal} in 1965, and later by Erd\H os, R\'enyi and S\'os \cite{ERS} in 1966 who gave the exact value of $ssat^{\Delta}_*(n,K_3)$ for $\Delta \geq (n+1)/2$ when $n$ is sufficiently large.
Later Hanson and Seyffarth \cite{Hanson} and Duffus and Hanson \cite{Duffus} studied saturated graphs with prescribed maximum and minimum degree, respectively.
Pach and Sur\'anyi \cite{Pach1,Pach2}  proved that for all $c \in (0,1]$ except certain sparse values the limit $\lim_{n \rightarrow \infty} \frac{ssat^{cn}(n,K_3)}{n}$ exists, and determined the limit when $\frac{1}{2} < c < 1$.
V\v r\v to and Zn\'am \cite{Vrto,Znam} extended the range of $c$ down to $\frac{3}{7}$.
F{\"u}redi \cite{Fur1} later surveyed these results and conjectured the leading term of the limit for $\frac{1}{3} < c < \frac{3}{7}$.

By \eqref{eq:1.1} we know that $ssat^{\Delta}(n,F) \leq sat^{\Delta}(n,F)$, so we can study the saturation number of $F$ indirectly via its semi-saturation number.
F{\"u}redi and Seress \cite{Fur2} determined $sat^{\Delta}(n,K_3)$ exactly when $\Delta \geq (n-2)/2$ and $n$ is sufficiently large.
They also showed that the limit $\lim_{n \rightarrow \infty} \frac{sat^{cn}(n,K_3)}{n}$ exists for all $c \in (0,1]$ except certain sparse values.
Later Erd\H os and Holzman \cite{EH} established the asymptotics of $sat^{cn}(n,K_3)$ as $n \rightarrow \infty$ for $\frac{2}{5} \leq c < \frac{1}{2}$.
Alon, Erd\H os, Holzman and Krivelevich \cite{Alon} proved analogous results on $K_r$ for $r \geq 4$. They showed that the limit $\lim_{n \rightarrow \infty} \frac{sat^{ cn}(n,K_r)}{n}$ exists for $r \geq 4$ and all $c \in (0,1]$ except certain sparse values.
Amin, Faudree, Gould and Sidorowicz \cite{Amin} studied $sat^{\Delta}(n,K_r)$ when $\Delta \leq n-2$.
Beyond these results, the problem of determining $ssat^{\Delta}(n,K_r)$ for $r \geq 4$ continues to resist solution and remains open.

To study $ssat^{\Delta}(n,K_r)$ for $r \geq 4$, we follow the approach of Pach and Sur\'anyi \cite{Pach1,Pach2} and its hypergraph reformulation due to Alon, Erd\H os, Holzman and Krivelevich \cite{Alon}.
A \textit{hypergraph} $\mathcal{H}$ on a finite vertex set $V$ is a family $\mathcal{H}=\{H_1,\ldots,H_m\}$ of nonempty subsets of $V$; we write $|\mathcal{H}|=m$ and $|H_i|$ for the size of a hyperedge.
A hypergraph $\mathcal{H}$ is \textit{$r$-uniform} if $|H_i|=r$ for every hyperedge $H_i$.
Two hyperedges $H,H' \in \mathcal{H}$ \textit{intersect} if $H \cap H' \neq \emptyset$; we say that $\mathcal{H}$ is \textit{$k$-intersecting} if $|H \cap H'| \geq k$ for all distinct $H,H' \in \mathcal{H}$. We also use $\mathcal{H}=(V,\mathcal{E})$ to denote a hypergraph with vertex set $V$ and hyperedge set $\mathcal{E}$.

For a hypergraph $\mathcal{H}=\{H_1, \ldots, H_m\}$, a \textit{fractional matching} of $\mathcal{H}$ is a nonnegative weight $(f_1,\ldots,f_m)$ of the hyperedges such that each vertex of $V$ is covered by total weight at most $1$.
The \textit{fractional matching number} $\nu^*(\mathcal{H})$ is the maximum total weight $\sum_{i=1}^m f_i$ over all fractional matchings; it is the optimum of a linear program dual to the fractional edge-cover problem, and is given explicitly in \eqref{eq:1.4} below.
An auxiliary quantity $a(\mathcal{H},c)$ is defined by a linear program in Definition~\ref{D11} below for each hypergraph $\mathcal{H}$ and a real parameter $c>0$.

\begin{definition}\label{D11} Given a hypergraph $\mathcal{H}=\{H_1, \ldots, H_m\}$ on a set $V$ and a real number $c >0 $, let $a(\mathcal{H},c)$ be the optimal value of the following linear programming problem.
	\begin{equation}\label{eq:1.3}
		\begin{aligned}
		   \min   \sum_{i=1}^{m} & \;  |H_i|\, y_i  \\
		\mathbf{s.t.}  \quad  \sum_{i:x \in H_i} y_i & \leq c \, , & & \forall \, x \in V \, , \\
		y_i &  \geq 0 \, ,& & \forall \, 1 \leq i \leq m \, , \\
		\sum_{i=1}^{m} y_i & =1 \, ,
		\end{aligned}
		\end{equation}
\end{definition}

The \textit{fractional matching number} $\nu^*(\mathcal{H})$ of the hypergraph $\mathcal{H}$ is defined as
\begin{equation}\label{eq:1.4}
	\begin{aligned}
	  \nu^*(\mathcal{H})  = \max   \sum_{i=1}^{m} & \;  f_i  \\
	\mathbf{s.t.}  \quad  \sum_{i:x \in H_i} f_i & \leq 1 \, , & & \forall \, x \in V \, , \\
	f_i &  \geq 0 \, ,& & \forall \, 1 \leq i \leq m \, . \\
	\end{aligned}
	\end{equation}
Clearly, $c \geq 1/\nu^*(\mathcal{H})$ is a necessary and sufficient condition for the feasibility of the linear programming problem \eqref{eq:1.3}. Note that if $\mathcal{H}$ is $r$-uniform and $c \geq 1/\nu^*(\mathcal{H})$, then $a(\mathcal{H},c) = r$.
\vspace{0.3em}

\begin{prop}\label{P12} \text{\bf \cite{Pach1}} For a fixed hypergraph $\mathcal{H}$, $a(\mathcal{H},c)$ is a continuous, convex, piecewise linear and monotone nonincreasing function on the interval $[1/\nu^*(\mathcal{H}), + \infty)$.
\end{prop}
\vspace{0.3em}

Given a positive integer $k$ and a real number $c >0 $, let
\begin{equation*}
A_k(c) : = \inf \{ \, a(\mathcal{H},c) : \mathcal{H} \text{ is $k$-intersecting, $ 1/\nu^*(\mathcal{H}) \leq c$} \, \} \, .
\end{equation*}
We show later in Section~2 that there exists a $k$-intersecting hypergraph $\mathcal{H}$ with $1/\nu^*(\mathcal{H}) \leq c$, which implies that $a(\mathcal{H},c)$ and $A_k(c)$ are well-defined. We can now state our main results.
\vspace{0.3em}

\begin{theorem}\label{T13}
For a fixed positive integer $k \geq 2$, $A_k(c)$ is a monotone nonincreasing, piecewise linear and right continuous function on $c \in (0,1]$. The points of discontinuity are all rational and contained in a sequence $1=c_0 >c_1 >c_2 > \ldots \rightarrow 0$.
\end{theorem}

Returning to the problem of determining $ssat^{\Delta}(n,K_r)$ for $r \geq 4$, we establish an asymptotically tight bound for $ssat^{\Delta}(n,K_r)$ using Theorem~\ref{T13}.

\begin{theorem}\label{T14}
For a fixed positive integer $r \geq 4$, a real number $0 < c < 1$ and $\Delta = cn$, if $A_{r-2}(c)$ is continuous at $c$, then
\[
\lim_{n \rightarrow \infty} \frac{ssat^{\Delta}(n,K_r)}{n} = A_{r-2}(c).
\]
\end{theorem}

We also determine $A_{r-2}(c)$ explicitly: we identify some breakpoints $c_i$ and give the exact value of $A_{r-2}(c)$ on each open interval between consecutive breakpoints.
\vspace{0.3em}

\begin{theorem}\label{T15}
For a fixed positive integer $r \geq 4$, we have\,:
\begin{equation*}
\begin{aligned}
    A_{r-2}(c) = \begin{cases}
    \quad r-1 & \text{if } \; \frac{r-1}{r} < c < 1 \, , \\
	\\
	\quad r-1+(r-2)(1-c) \quad  \quad &  \text{if } \; \frac{r-3/2}{r-1/2} < c < \frac{r-1}{r} \, , \\
	\\
	\quad 3r-4-(2r-2)c & \text{if } \; \frac{r-2}{r-1} < c < \frac{r-3/2}{r-1/2} \, , \\
	\\
    \quad r & \text{if } \; \frac{r}{r+2} < c < \frac{r-2}{r-1} \, .
\end{cases}
\end{aligned}
\end{equation*}
\end{theorem}

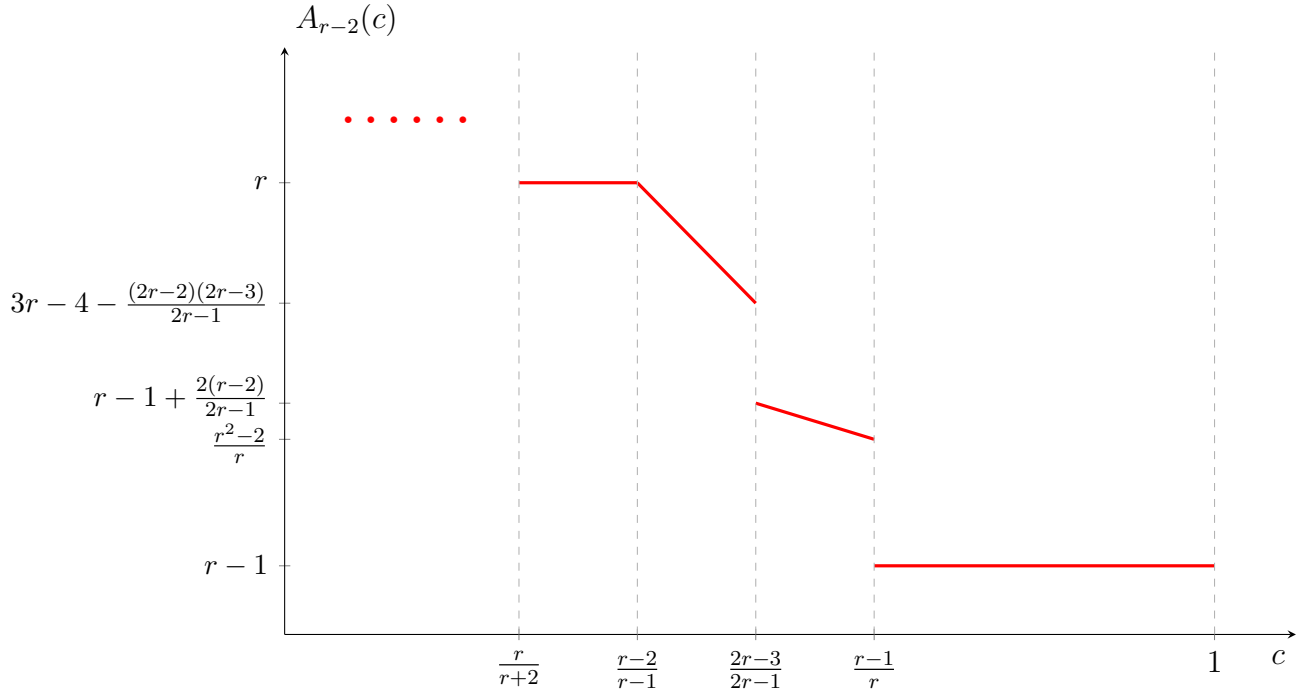
\begin{figure}[htbp]
\centering
\hspace*{-4em}
\begin{tikzpicture}

  \pgfmathsetmacro{\RR}{5}
  \pgfmathsetmacro{\YC}{4.75}
  \pgfmathsetmacro{\yGap}{0.7}

  \pgfmathsetmacro{\ySegIIIshift}{0.22}
  \pgfmathsetmacro{\yRtick}{\YC + 3*(\RR - \YC)}
  \pgfmathsetmacro{\yRsqtick}{\YC + 3*((\RR*\RR-2)/\RR - \YC) - \ySegIIIshift}
  \pgfmathsetmacro{\yLast}{\yRsqtick - \yGap}
  \pgfmathsetmacro{\yRmidtick}{\YC + 3*((\RR-1)+2*(\RR-2)/(2*\RR-1) - \YC) - \ySegIIIshift}
  \pgfmathsetmacro{\yRctick}{\YC + 3*((3*\RR-4-(2*\RR-2)*(2*\RR-3)/(2*\RR-1)) - \YC)}
\begin{axis}[
  width=0.88\textwidth,
  height=0.55\textwidth,
  xlabel={},
  ylabel={},
  axis lines=left,
  clip=false,
  xmin=-0.95,
  xmax=3.15,
  ymin=3.0,
  ymax=6.25,
  grid=none,
  minor x tick num=0,
  minor y tick num=0,
  enlarge x limits=false,
  enlarge y limits=false,
  xtick={0,0.48,0.96,1.44,2.82},
  xticklabels={$\frac{r}{r+2}$,$\frac{r-2}{r-1}$,$\frac{2r-3}{2r-1}$,$\frac{r-1}{r}$,$1$},
  xticklabel style={font=\normalsize},
  ytick={\yLast, \yRsqtick, \yRmidtick, \yRctick, \yRtick},
  yticklabels={%
    {{$r-1$}},%
    {\raisebox{-3ex}{$\frac{r^2-2}{r}$}},%
    \raisebox{2ex}{$r-1+\frac{2(r-2)}{2r-1}$},$3r-4-\frac{(2r-2)(2r-3)}{2r-1}$,$r$},
  yticklabel style={font=\small},
  after end axis/.code={%
    \node[anchor=north west] at (axis description cs:0.045,0.9) {\Large\textcolor{red}{\textbf{\ldots\ldots}}};
    \node[font=\normalsize] at (axis description cs:0,1) [anchor=south west] {$A_{r-2}(c)$};
    \node[font=\normalsize] at (axis description cs:1,0) [anchor=north east, yshift=-2pt] {$c$};
  },
]

  \pgfmathsetmacro{\brkA}{\RR/(\RR+2)}
  \pgfmathsetmacro{\brkB}{(\RR-2)/(\RR-1)}
  \pgfmathsetmacro{\brkC}{(2*\RR-3)/(2*\RR-1)}
  \pgfmathsetmacro{\brkD}{(\RR-1)/(\RR)}
  \pgfmathsetmacro{\wseg}{0.48}
  \pgfmathsetmacro{\wR}{1.38}
  \pgfmathsetmacro{\xA}{0}
  \pgfmathsetmacro{\xB}{\wseg}
  \pgfmathsetmacro{\xC}{2*\wseg}
  \pgfmathsetmacro{\xD}{3*\wseg}
  \pgfmathsetmacro{\xE}{3*\wseg+\wR}

  \pgfmathsetmacro{\yLo}{3.0}
  \pgfmathsetmacro{\yHi}{6.25}
  \addplot[very thick, red, domain=\xA:\xB, samples=2] {\YC + 3*(\RR - \YC)};
  \addplot[very thick, red, domain=\xB:\xC, samples=50]
    {\YC + 3*((3*\RR-4-(2*\RR-2)*(\brkB+((x-\xB)/\wseg)*(\brkC-\brkB))) - \YC)};
  \addplot[very thick, red, domain=\xC:\xD, samples=50]
    {\YC + 3*(((\RR-1)+(\RR-2)*(1-(\brkC+((x-\xC)/\wseg)*(\brkD-\brkC)))) - \YC) - \ySegIIIshift};
  \addplot[very thick, red, domain=\xD:\xE, samples=2] {\yLast};
  \addplot[forget plot, dashed, gray!60] coordinates {(\xA,\yLo)(\xA,\yHi)};
  \addplot[forget plot, dashed, gray!60] coordinates {(\xB,\yLo)(\xB,\yHi)};
  \addplot[forget plot, dashed, gray!60] coordinates {(\xC,\yLo)(\xC,\yHi)};
  \addplot[forget plot, dashed, gray!60] coordinates {(\xD,\yLo)(\xD,\yHi)};
  \addplot[forget plot, dashed, gray!60] coordinates {(\xE,\yLo)(\xE,\yHi)};

\end{axis}
\end{tikzpicture}
\caption{A schematic plot of $A_{r-2}(c)$ from Theorem~\ref{T15}.}

\end{figure}
\vspace{0.5em}

For some specific $\Delta$, we have the following results.

\begin{theorem}[Erd\H os, R\'enyi and S\'os \cite{ERS}]\label{T153}
For every $n \geq 4$, we have $ssat^{n-2}(n,K_3)=2n-5$.
\end{theorem}

\begin{theorem}\label{T16} For $r \geq 4$ and $n$ sufficiently large, we have $ ssat^{\Delta}(n,K_r) =  (r-1)n-\frac{r(r-1)}{2}$ for any  $n-1-\lfloor \frac{n-r}{r}\rfloor \leq \Delta \leq n-2$ except for $(r,\Delta) =(4,n-2)$, in which we have $ ssat^{\Delta}(n,K_r) =  (r-1)n-\frac{r(r-1)}{2}-1$.

\end{theorem}
Recall that  $ssat_*^{n-1}(n,K_r)=ssat^{n-1}(n,K_r)=ssat(n,K_r)=sat(n,K_r)=(r-2)n-\frac{(r-2)(r-3)}{2}$ (\cite{EHM}). Combining with Theorem \ref{T16}, we immediately have that $c_0=1$ is a discontinuous point in Theorem~\ref{T13}.

A natural question is whether one can determine $sat(n,K_r \vee F)$ from the saturation number of $F$.
An easy observation is that $sat(n,K_r \vee F) \leq r(n-r)+sat(n-r,F)+ \binom{r}{2}$ for all $n > |V(F)| \, +r-1$, as an $F$-saturated graph $G$ on $n-r$ vertices can be extended to a $(K_r \vee F)$-saturated graph on $n$ vertices.
An important problem is to find all pairs $(r,F)$ for which the equality above holds.

\begin{prob}\label{P17} Determine all pairs $(r,F)$ such that $sat(n,K_r \vee F) = r(n-r)+sat(n-r,F)+ \binom{r}{2}$ holds for all sufficiently large $n$.
\end{prob}
\begin{table}[H]
	\centering
	\small
	\setlength{\tabcolsep}{6pt}
	\renewcommand{\arraystretch}{1.25}
	\caption{Some confirmed cases related to Problem~\ref{P17}.}
	\label{tab:P17-known}
	\begin{tabular}{@{}p{6.2cm}p{9.6cm}@{}}
		\hline
		\textbf{Reference} & \textbf{Verified case $(r,F)$} \\
		\hline
		Hu, Ji and Cui \cite{HU1} & $(1,P_k)$ for $k \geq 5$. \\
		Hu, Luo and Peng  \cite{HU} & $(1,F)$ where $F=K_2 \cup F'$ for some graph $F'$. \\
		Qiu, He, Lu and Xu \cite{Qiu2026} & $(1,C_k)$ for $k \geq 8$. \\
		Song, Hu, Ji and Cui \cite{Song} & $(1,C_4)$. \\
		Zhang, Cui, Hu, Yue and Ji \cite{Zhang2026} & $(1,F)$ where $F$ is a linear forest. \\
		Zhang, You and Zhao \cite{ZZY2025} & $(2,P_k)$ for $k \geq 3$. \\
		\hline
	\end{tabular}
\end{table}

Recently many authors have confirmed that the equality above holds for several specific $(r,F)$; see Table~\ref{tab:P17-known}.
Nevertheless, this problem is still far from being completely settled, and many cases remain open.
In this paper, by applying Theorem~\ref{T16}, we are able to establish the equality for a large class of pairs $(r,F)$. We call a vertex $v$ \emph{conical} if $d(v) = n-1$.

\begin{theorem}\label{T18} Fix integers $r \geq 1$ and $t \geq 2$. If $F$ is a graph without isolated vertex such that every edge of $F$ belongs to a $K_t$ and $sat(n,F) < (t-1)n - \binom{t}{2}-2$ for sufficiently large $n$, then $sat(n,K_r \vee F) = r(n-r)+sat(n-r,F)+ \binom{r}{2}$. Moreover, every extremal graph $G$ for $sat(n,K_r \vee F)$ contains $r$ conical vertices.
\end{theorem}

The results in \cite{HU1,HU,Zhang2026,ZZY2025} are immediate corollaries of Theorem \ref{T18} due to the known saturation numbers of paths, linear forests, etc. Moreover, the equality in Problem \ref{P17} holds for all $(r,F)$ where $F$ is a linear forest without isolated vertices, by applying Theorem \ref{T18} taking $t=2$.
For every $t \geq 2$, we construct a graph $F$ such that the condition of Theorem~\ref{T18} holds.

	The remainder of this paper is organized as follows. In Section~2, we present notation and preliminary results and prove Theorems~\ref{T13} and~\ref{T14}. In Section~3, we prove Theorem~\ref{T15}. In Section~4, we prove Theorems~\ref{T16} and~\ref{T18}.

\section{On $k$-intersecting hypergraphs}

In this section we present some notation and preliminary results for the proofs of Theorems \ref{T13} and \ref{T14}.
The main idea of the proofs of these results comes from Alon, Erd\H{o}s, Holzman and Krivelevich \cite{Alon}.

The construction below shows that there exists a $k$-intersecting hypergraph $\mathcal{H}$ such that the constraint system in \eqref{eq:1.3} is feasible.
\vspace{0.5em}

\noindent\textbf{Construction \cite{Alon}.}
Fix a positive integer $k$ and a real number $c>0$. Let $q$ be a prime power such that $c \geq (q+1)/(q^2+q+1)$.
Let $V$ consist of $k$ copies of two disjoint sets of size $q^2+q+1$, namely
$A_1,B_1,\ldots,A_k,B_k$.
Identify the vertices of $A_i$ with the points of the projective plane $PG(2,q)$, and identify
the vertices of $B_i$ with the lines of $PG(2,q)$.
Define a hypergraph $\mathcal{H}$ on $V$ as follows: for each line $\ell$ of $PG(2,q)$, let $H_\ell$
be the hyperedge consisting of all points of $\ell$ in each $A_i$ together with the singleton $\{\ell\}$ in
each $B_i$.
\vspace{0.5em}

Clearly, $\mathcal{H}$ has $q^2+q+1$ hyperedges and each hyperedge has size $|H_\ell|=k(q+1)+k=k(q+2)$, $\mathcal{H}$ is $k$-intersecting, and setting $y_i=1/(q^2+q+1)$ for every hyperedge $H_\ell$
yields a feasible solution of \eqref{eq:1.3}.
The construction implies that $A_k(c)$ is well-defined for all $c>0$. We now present some basic properties of $A_k(c)$ for all $c\in (0,1]$.

\begin{lemma}\label{L21}
$A_k(c) \leq 2k(1+1/c)$ for all $0<c\le 1$.
\end{lemma}
\noindent\textbf{Proof. }Let $p$ be a prime satisfying $1/c \leq p \leq 2/c$. Taking the hypergraph $\mathcal{H}$ in the construction, we have $A_k(c) \leq a(\mathcal{H},c) = k(p+2) \leq k(2+2/c) = 2k(1+1/c) $. \q
\vspace{0.5em}

\begin{lemma}\label{L22}
	The  infimum of $A_k(c)$ in the definition may be achieved by a hypergraph with at most $2k(1/c+1/c^2)+1$ hyperedges  for all $0<c\le 1$.
\end{lemma}
\noindent\textbf{Proof. }Let $\mathcal{H}=\{H_1,\ldots,H_m\}$ be a $k$-intersecting hypergraph on the vertex set $V$ such that $a(\mathcal{H},c) \leq  2k(1+1/c)$ by Lemma \ref{L21}.
The system of restrictions in \eqref{eq:1.3} defines a convex, bounded and non-empty polytope $P$ in $\mathbb{R}^m$.
Hence the objective function of \eqref{eq:1.3} attains its minimum at a vertex $u$ of $P$.
Since $u$ is a vertex of $P$, it must be the intersection of at least $m$ hyperplanes corresponding to the constraints in \eqref{eq:1.3} which are of 3 types:
\begin{enumerate}
	\item $ \sum_{i:x \in H_i} y_i = c$ for some $x \in V$,
	\item $y_i =0$ for some $i \in [m]$,
	\item $\sum_{i=1}^m y_i = 1$.
\end{enumerate}
Since
\begin{equation*}
	\sum_{x \in V} \sum_{i:x \in H_i} y_i = \sum_{i=1}^m |H_i| y_i = a(\mathcal{H},c) \leq 2k(1+1/c)\,,
\end{equation*}
at most $2k(1/c+1/c^2)$ hyperplanes of type (i) contain $u$.
Therefore there are at least $m-2k(1/c+1/c^2)-1$ values of $i$ such that $y_i=0$ occurs at $u$.
Let $\mathcal{H}'$ be the hypergraph obtained from $\mathcal{H}$ by removing the hyperedges $H_i$ for which $y_i=0$ at $u$. Clearly $\mathcal{H}'$ is $k$-intersecting, $a(\mathcal{H}',c) = a(\mathcal{H},c) $ and $|\mathcal{H}'|\, \leq 2k(1/c+1/c^2)+1$. \q
\vspace{0.5em}

A hypergraph $\mathcal{H}=(V,\mathcal{E})$ is called \emph{separated} if for any $x \neq y \in V$, there exists a hyperedge $H \in \mathcal{E}$ such that $|H \cap \{x,y\}| \, = 1$.
It is obvious that for any $x \neq y \in V$, the sets of hyperedges containing $x$ and $y$, respectively, are different, and therefore $|V| \, \leq 2^{|\mathcal{E}|}$ if $\mathcal{H}=(V,\mathcal{E})$ is separated.
For any hypergraph $\mathcal{H}'$, there exists a separated hypergraph $\mathcal{H}$ obtained from $\mathcal{H}'$ by identifying the vertices.
Clearly $\mathcal{H}$ has the same number of hyperedges and the same fractional matching number as $\mathcal{H}'$.
If $V(\mathcal{H})=\{x_1,\ldots,x_r\}$ and $x_i$ is obtained by identifying $a_i$ vertices of $\mathcal{H}'$, we say that $\mathcal{H}'$ is an $(a_1,\ldots,a_r)$-\textit{blow-up} of $\mathcal{H}$. Let
\begin{equation*}
	\mathscr{B}(\mathcal{H}) = \{\,(a_1,\ldots,a_r) \in \mathbb{N}^r : \text{the $(a_1,\ldots,a_r)$-blow-up of $\mathcal{H}$ is $k$-intersecting\,} \}\,.
\end{equation*}
For given $c>0$, define a family of hypergraphs by
\begin{equation*}
	\mathscr{H}(c) = \{\text{$\mathcal{H}$ is separated : $\nu^*(\mathcal{H}) \geq 1/c$, $|\mathcal{E}(\mathcal{H})|\, \leq 2k(1/c+1/c^2)+1$, $\mathscr{B}(\mathcal{H}) \neq \emptyset$\,} \}\,.
\end{equation*}
By the construction, $\mathscr{H}(c) \neq \emptyset$.

\begin{lemma}\label{L23}
	The infimum of $A_k(c)$ in the definition is attained  for all $0<c\le 1$.
\end{lemma}
\noindent\textbf{Proof. }By Lemma~\ref{L22}, we may restrict attention to hypergraphs with at most $2k(1/c+1/c^2)+1$ hyperedges.
Let the family of such hypergraphs be $\mathscr{H}_1(c)$.
Let $\mathscr{H}_2(c)$ be the family obtained from $\mathscr{H}_1(c)$ by identifying vertices to obtain separated hypergraphs.
Then every $\mathcal{H}\in \mathscr{H}_1(c)$ is an $(a_1,\ldots,a_r)$-blow-up of some $\mathcal{H}'\in \mathscr{H}_2(c)$ for some $r \leq r^* := 2^{2k(1/c+1/c^2)+1}$.
Therefore $\mathscr{H}_2(c) \subseteq \mathscr{H}^{r^*}(c)$, where $\mathscr{H}^{r^*}(c) \subseteq \mathscr{H}(c)$ is the family of separated hypergraphs with at most $r^*$ vertices.
Rewrite \eqref{eq:1.3} as
\begin{equation*}
	\begin{aligned}
		A_k(c)
		=
		\min_{\mathcal{H}\in \mathscr{H}^{r^*}(c)} &
		\inf_{(a_1,\ldots,a_r)\in \mathscr{B}(\mathcal{H})}  &  &
		\min \sum_{i=1}^{m}  \; \left(\sum_{x_j\in H_i} a_j\right)y_i \\
		\mathbf{s.t.}  \quad  \sum_{i:x_j \in H_i} y_i
		& \leq c \, , & & \forall \, x_j \in V(\mathcal{H}')  \, , \\
		y_i &  \geq 0 \, , &  & \forall \, 1 \leq i \leq m   \, , \\
		\sum_{i=1}^{m} y_i & =1 \,,
	\end{aligned}
\end{equation*}where $m$ is the size of $\mathcal{H}$ with $m\le 2k(1/c+1/c^2)+1$.
Let the natural partial order $\prec$ on $\mathbb{N}^r$ be defined by $(a_1,\ldots,a_r) \prec (b_1,\ldots,b_r)$ if and only if $a_i \leq b_i$ for every $i \in [r]$.
As the poset $(\mathbb{N}^r, \prec)$ has no infinite antichain for all $r \leq r^*$, the minimal elements of $\mathscr{B}(\mathcal{H})$ are finite and the infimum is attained. \q
\vspace{0.5em}

\noindent\textbf{Proof of Theorem~\ref{T13}. }By Lemma~\ref{L23}, for every fixed $c^*>0$ the value of $A_k(c)$ on $[c^*,1]$ is determined by a finite number of blow-ups of finitely many separated hypergraphs.
Therefore $A_k(c)$ is the minimum of finitely many functions $a(\mathcal{H},c)$ on $[c^*,1]$.
By Proposition~\ref{P12}, $A_k(c)$ is monotone nonincreasing, piecewise linear and has only possible left-discontinuities in $[c^*,1]$ of the form $1/\nu^*(\mathcal{H})$ for some hypergraph $\mathcal{H}$ from this finite collection; these points are all rational. \q
\vspace{0.5em}

A hypergraph $\mathcal{S}$ with hyperedge set $\mathcal{E}$ is a \emph{sunflower} if
\[
E\cap E'=\bigcap_{F\in\mathcal{E}} F
\quad\text{for all distinct } E,E'\in\mathcal{E}\,;
\]
the sets $E\setminus \bigcap_{F\in\mathcal{E}} F$ are called the \emph{petals}.

\begin{lemma}\label{L24}
Fix an integer $r \geq 4$ and a real number $C>0$. Then there exists $n_0$ such that for every $n > n_0$, if $G = (V,E)$ is $K_r$-semi-saturated on $n$ vertices with $e(G) \le Cn$, then there is a set $V_0 \subseteq V$ with the following properties.
\begin{enumerate}
\item[(a)] $|V_0| \le (2C+1)n/\log\log n$.
\item[(b)]  Let $\mathcal{H}$ be the hypergraph on vertex set $V_0$ whose hyperedge multiset is $\{H(x): x\in V\setminus V_0\}$ (each $x$ contributes one hyperedge), where $H(x):=N_G(x)\cap V_0$. Then $\mathcal{H}$ is $(r-2)$-intersecting.

\end{enumerate}
\end{lemma}

\noindent\textbf{Proof. }
Let $X:=\{x\in V : d_G(x)\ge \log\log n\}$. As $\sum_{x\in V} d_G(x)=2e(G)\le 2Cn$,
\begin{equation}\label{eq:L24-X}
|X|\le \frac{2Cn}{\log\log n}.
\end{equation}
For $y\in V\setminus X$, set $H_X(y):=N_G(y)\cap X$. Then $|H_X(y)|<\log\log n$. Define
\[
Y:=\bigl\{y\in V\setminus X:\exists\, z\in V\setminus X \text{ with } |H_X(y)\cap H_X(z)|<r-2\bigr\},
\qquad V_0:=X\cup Y.
\]
Then (a) follows from \eqref{eq:L24-X} once we show $|Y|\le n/\log\log n$.

Fix distinct $u,v\in V\setminus V_0$. Since $u,v\notin Y$,  $|H_X(u)\cap H_X(v)|\ge r-2$. Since $X\subseteq V_0$, we have $H_X(u)\cap H_X(v)\subseteq H_{V_0}(u)\cap H_{V_0}(v)$. Therefore $|H_{V_0}(u)\cap H_{V_0}(v)|\ge r-2$, proving~(b).

It remains to bound $|Y|$. Let $\mathcal{F}:=\{H_X(y):y\in Y\}$, viewed as a hypergraph on vertex set $X$.

\begin{claim}\label{clm:L24-sunflower}
$\mathcal{F}$ contains no sunflower with more than $M+1$ petals, where $M:=\lfloor(\log\log n)^2+\log\log n\rfloor$.
\end{claim}
\noindent\textbf{Proof of Claim \ref{clm:L24-sunflower}.}
This follows the sunflower argument of \cite{Alon}, adapted to $K_r$-semi-saturation and to our definition of $Y$.
Suppose that $\{H_X(y_i)\}_{i=0}^{M}$ is a sunflower of $\mathcal{F}$ on distinct vertices $y_i\in Y$, and write $U:=\bigcap_{i=0}^{M} H_X(y_i)$ for its core.
Since $y_0\in Y$, there is $z\in V\setminus X$ such that
\begin{equation}\label{eq:L24-badpair}
|H_X(y_0)\cap H_X(z)|<r-2.
\end{equation}
Every vertex of $V\setminus X$ has degree strictly less than $\log\log n$ in $G$, hence also in $G[V\setminus X]$.
Thus fewer than $(\log\log n)^2$ vertices of $V\setminus X$ are of distance at most two from $z$ in $G[V\setminus X]$.
Since $M+1>(\log\log n)^2+\log\log n$, among the $M+1$ vertices, there are more than $\log\log n$ vertices, say $y_{1},\ldots,y_{s}$ with $s>\log\log n$,
such that $y$ and $z$ have  distance in $G[V\setminus X]$ strictly larger than two for any $y\in \{y_{1},\ldots,y_{s}\}$.
Then $y_{i} z\notin E(G)$ and $N_G(y_{i})\cap N_G(z)\cap (V\setminus X)=\emptyset$ for all $1\le i\le s$. Hence for all $1\le i\le s$
\begin{equation}\label{eq:L24-cmn}
N_G(y_{i})\cap N_G(z)\subseteq X.
\end{equation}
Since $G$ is $K_r$-semi-saturated, $G+y_{i} z$ contains a copy of $K_r$ for all $1\le i\le s$.
Let $T_{i}$ be the vertex set of the clique for all $1\le i\le s$. Then $|T_{i}|=r$ and $y_{i},z\in T_{i}$. Let $T_{i}'=T_{i}\setminus\{y_{i},z\}$ for all $1\le i\le s$. Then $T_{i}'\subseteq  N_G(y_{i})\cap N_G(z)\subseteq X$ by \eqref{eq:L24-cmn}.

If $T_{i}'\subseteq U$ for some $i$, then $T_{i}'\subseteq H_X(y_0)$ by $U\subseteq H_X(y_0)$, and $T_{i}'\subseteq H_X(z)$. Hence $T_{i}'\subseteq H_X(y_0)\cap H_X(z)$ and $|H_X(y_0)\cap H_X(z)|\ge r-2$ which is a contradiction to \eqref{eq:L24-badpair}.
Thus $T_{i}'\not\subseteq U$ for any $1\le i\le s$, and we may choose $x_i\in T_{i}'\setminus U$.
For $1\le j\neq i\le s$, we have $H_X(y_i)\cap H_X(y_j)=U$. Since  $x_i\notin U$, we have $x_i\notin H_X(y_j)$.
The vertices $x_1,\ldots x_s$ are pairwise distinct as they lie in pairwise disjoint petals, and each lies in $N_G(z)\cap X=H_X(z)$.
Therefore $|H_X(z)|\ge s>\log\log n$,  a contradiction to $z\in V\setminus X$.
$\hfill\square$
\vspace{0.5em}

Recall each hyperedge of $\mathcal{F}$ has size at most $\lfloor\log\log n\rfloor$.
By a theorem of Erd\H{o}s and Rado \cite{ErdosRado}, if a hypergraph has more than $t!\, \mu^{t}$ hyperedges of size at most $t$, then it contains a sunflower with $\mu+1$ petals.
With $t:=\lfloor\log\log n\rfloor$ and $\mu:=M$,  Claim \ref{clm:L24-sunflower} implies $|\mathcal{F}|\le t!\, M^{t}$. So for sufficiently large $n$, we have
\begin{equation}\label{eq:L24-ER}
|\mathcal{F}|
\le (\log\log n)!\,\left((\log\log n)^2+\log\log n\right)^{\log\log n}
=o\left(\frac{n}{(\log\log n)^3}\right).
\end{equation}

Now let $S\subseteq X$ such that $\mathcal{A}_S:=\{y\in Y:H_X(y)=S\}\neq\emptyset$.
Choose $w\in V\setminus X$ such that $|S\cap H_X(w)|<r-2$.
We claim that every $y\in \mathcal{A}_S$ satisfies $\mathrm{dist}_{G[V\setminus X]}(y,w)\le 2$.
Otherwise $yw\notin E(G)$ and $N_G(y)\cap N_G(w)\cap (V\setminus X)=\emptyset$. So $N_G(y)\cap N_G(w)\subseteq X$.
Since $G+yw$ creates a $K_r$, there is an $(r-2)$-clique inside $N_G(y)\cap N_G(w)\subseteq X$ which gives $|H_X(y)\cap H_X(w)|\ge r-2$, a contradiction with $|S\cap H_X(w)|=|H_X(y)\cap H_X(w)|<r-2$.
Thus $\mathcal{A}_S$ lies in a ball of radius 2 around $w$ in $G[V\setminus X]$, which has $O((\log\log n)^2)$ vertices, whence $|\mathcal{A}_S|=O((\log\log n)^2)$.

Together with \eqref{eq:L24-ER}, we have
\[
|Y|\le |\mathcal{F}|\cdot O\bigl((\log\log n)^2\bigr)=o\Bigl(\frac{n}{\log\log n}\Bigr).
\]
So  $|Y|\le n/\log\log n$ for $n$ large enough. \q
\vspace{0.5em}

\begin{lemma}\label{L25}
Let $r\ge 4$ and $c'\in(0,1)$. Suppose $\mathcal{H}$ is an $(r-2)$-intersecting hypergraph such that \eqref{eq:1.3} with parameter $c'$ is feasible and $a(\mathcal{H},c')=A_{r-2}(c')$.
Then for every sufficiently large integer $n$ there exists a $K_r$-semi-saturated graph $G$ on $n$ vertices such that
\[
\Delta(G)\le c' n+O(1),
\qquad
e(G)\le A_{r-2}(c')\, n+O(1).
\]
\end{lemma}

\noindent\textbf{Proof. }
By Lemmas~\ref{L22} and~\ref{L23}, the value $A_{r-2}(c')$ is attained and we may fix a minimizing $(r-2)$-intersecting $\mathcal H$ with $a(\mathcal H,c')=A_{r-2}(c')$ such that $|\mathcal H|$ and $|V(\mathcal H)|$ are both $O(1)$ as $n\to\infty$ (constants depending only on $r$ and $c'$).

Write $\mathcal{H}=\{H_1,\ldots,H_m\}$ on $V(\mathcal{H})$, let $\{y_i\}_{i=1}^m$ be optimal for \eqref{eq:1.3} at $c'$, and put $w_i:=y_i/c'$, $T:=\sum_{i=1}^m w_i=1/c'$ (so $\{w_i\}_{i=1}^m$ is a fractional edge packing of total weight~$T$). Now we construct a  $K_r$-semi-saturated graph $G$ on $n$ vertices.
For sufficiently large integer $n$, let $V(G)=X\sqcup V_0\sqcup V_1\sqcup\cdots\sqcup V_m$ with $X= V(\mathcal{H})$, $|V_i|=\lfloor (w_i/T)(n-|X|)\rfloor$ for $1\le i\le m$, and $|V_0|=n-|X|-\sum_{i=1}^m|V_i|$ (so $|V_0| \, \le m$).
Let $G[X\cup V_0]$ be a clique, and for each $i\in \{1,\ldots,m\}$, join every vertex of $V_i$ to all of $H_i\subseteq X$ and to no other vertex of~$X$ (no further edges between distinct sets $V_i,V_j$ with $i\neq j$, nor between $\bigcup_{i=1}^m V_i$ and~$V_0$). Not difficult to check that $G$ is $K_r$-semi-saturated.

Note that for every $x\in X$, $\sum_{i:\, x\in H_i}|V_i|\le (n-|X|)/T$. So we obtain $\Delta(G)\le (n-|X|)/T+|X|+|V_0|\le c'n+O(1)$ by $|X|=O(1)$ and $m=O(1)$.
Also,
\[
e(G)\le \binom{|X|+|V_0|}{2}+\sum_{i=1}^m |V_i|\,|H_i|
\le \Bigl(\sum_{i=1}^m |H_i|y_i\Bigr)n+O(1)=A_{r-2}(c')\,n+O(1).
\]
So we are done.\q

\vspace{0.5em}

\noindent\textbf{Proof of Theorem~\ref{T14}. }
Let $\Delta = cn$. We first bound $\limsup_{n \rightarrow \infty} ssat^{\Delta}(n,K_r)/n$. Suppose, towards a contradiction, that $\limsup_{n \rightarrow \infty} ssat^{\Delta}(n,K_r)/n \ge A_{r-2}(c)+\varepsilon$ for some $\varepsilon>0$. Since $A_{r-2}(c)$ is continuous at $c$, choose $\delta>0$ such that $A_{r-2}(c-\delta)<A_{r-2}(c)+\varepsilon$. Take an $(r-2)$-intersecting hypergraph $\mathcal{H}$ with $a(\mathcal{H},c-\delta)=A_{r-2}(c-\delta)$ (feasibility of \eqref{eq:1.3} at $c-\delta$ is part of the definition of $A_{r-2}(c-\delta)$).
Set $c'=c-\delta$. By Lemma~\ref{L25}, for sufficiently large $n$, there is a $K_r$-semi-saturated graph $G$ on $n$ vertices with $\Delta(G)\le (c-\delta)n+O(1)$ and $e(G)\le A_{r-2}(c-\delta)n+O(1)$. Since  $n$ is sufficiently large, we  have $(c-\delta)n+O(1)\le cn$. Hence $\Delta(G)\le \Delta$ and therefore
\[
ssat^{\Delta}(n,K_r)\le e(G)\le A_{r-2}(c-\delta)n+O(1).
\]
Let $n\to\infty$, we have $\limsup_{n\to\infty} ssat^{\Delta}(n,K_r)/n\le A_{r-2}(c-\delta)<A_{r-2}(c)+\varepsilon$, a contradiction.

For the lower bound, suppose $\liminf_{n\rightarrow\infty} ssat^{\Delta}(n,K_r)/n\le A_{r-2}(c)-\varepsilon$ for some $\varepsilon>0$.
Choose a strictly increasing sequence $\{n_i\}_{i=1}^{\infty}$ and graphs $G^i$ that are $K_r$-semi-saturated with $|V(G^i)|=n_i$, $\Delta(G^i)\le cn_i$, and $e(G^i)\le(A_{r-2}(c)-\varepsilon)n_i$. Pick $\delta>0$ with $A_{r-2}(c+\delta)>A_{r-2}(c)-\varepsilon$. Let $C:=A_{r-2}(c)-\varepsilon$.
By
Lemma~\ref{L24}, for each $i$, there is  $V_0^i\subseteq V(G^i)$ with $|V_0^i|=o(n_i)$ and an $(r-2)$-intersecting hypergraph $\mathcal{H}^i$ on $V_0^i$ from the neighbourhoods of vertices in $V(G^i)\setminus V_0^i$, as in part~(b). For each $i$, define weights on $\mathcal{E}(\mathcal{H}^i)$ by
\[
w(H)=\frac{|\{x\in V(G^i)\setminus V_0^i : N_{G^i}(x)\cap V_0^i = H\}|}{|V(G^i)\setminus V_0^i|}.
\]
Then $\sum_{H\in \mathcal{E}(\mathcal{H}^i)} w(H)=1$. Since $|V(G^i)\setminus V_0^i|=(1-o(1))n_i$, for $z\in V_0^i$ and $H\in \mathcal{E}(\mathcal{H}^i)$, we have
\[
\sum_{z \in H} w(H)=\frac{|N_{G^i}(z)\cap (V(G^i)\setminus V_0^i)|}{|V(G^i)\setminus V_0^i|}\le \frac{d_{G^i}(z)}{|V(G^i)\setminus V_0^i|}\le \frac{cn_i}{|V(G^i)\setminus V_0^i|}\le c+\delta
\]
for all sufficiently large $n_i$. Thus $w$ is feasible for \eqref{eq:1.3} at $c+\delta$, where $\sum_{H\in \mathcal{E}(\mathcal{H}^i)} |H|w(H)\ge a(\mathcal{H}^i,c+\delta)\ge A_{r-2}(c+\delta)$. Counting edges between $V_0^i$ and $V(G^i)\setminus V_0^i$ gives
\[
\sum_{x\in V(G^i)\setminus V_0^i} |N_{G^i}(x)\cap V_0^i| = |V(G^i)\setminus V_0^i|\sum_{H\in \mathcal{E}(\mathcal{H}^i)} |H|w(H).
\]
Hence
\[
e(G^i)\ge \sum_{x\in V(G^i)\setminus V_0^i} |N_{G^i}(x)\cap V_0^i|
\ge (1-o(1))n_i\, A_{r-2}(c+\delta)
>(A_{r-2}(c)-\varepsilon)n_i
\]
for all sufficiently large $i$, a contradiction. \q
\vspace{0.5em}

\section{The extremal function $A_{r-2}(c)$}

In this section we prove Theorem~\ref{T15}. Assume $r\ge 4$.

\subsection{Upper bounds}\label{subsec:T15-upper}

Theorem~\ref{T15} will be proved in Subsection~\ref{subsec:T15-lb} by matching upper and lower bounds.
We first establish the upper bounds in the next proposition; the proof relies entirely on the explicit graph constructions in Appendix~\ref{app:T15-constructions}.

\begin{prop}\label{prop:T15-upper}
Fix an integer $r\ge 4$. For every sufficiently large $n$,

there exists an $n$-vertex graph $G$ with $e(G)=f(r,c)n+O(1)$ which is $K_r$-semi-saturated and satisfies $\Delta(G)\le cn$, where
\begin{equation*}
f(r,c) : =\begin{cases}
\;r-1& \textbf{(I)} \;\; \text{if } \; \frac{r-1}{r} < c < 1 \, , \\[0.4em]
\;2r-3-(r-2)c& \textbf{(II)} \;\; \text{if } \; \frac{r-3/2}{r-1/2} < c < \frac{r-1}{r} \, , \\[0.4em]
\;3r-4-(2r-2)c \quad \; & \textbf{(III)} \;\; \text{if } \; \frac{r-2}{r-1} < c < \frac{r-3/2}{r-1/2} \, , \\[0.4em]
\;r& \textbf{(IV)} \;\; \text{if } \; \frac{r}{r+2} < c < \frac{r-2}{r-1} \, .
\end{cases}
\end{equation*}
\end{prop}
\noindent\textbf{Proof. }See Appendix~\ref{app:T15-constructions}.

\subsection{Lower bounds}\label{subsec:T15-lb}

It remains to prove the matching lower bound: every $K_r$-semi-saturated graph $G$ on $n$ vertices with $\Delta(G)\le cn$ satisfies $e(G)\ge f(r,c)n-o(n)$, where $f(r,c)$ as defined in Proposition~\ref{prop:T15-upper}.

Fix $\frac{r}{r+2}<c<1$ and $c \neq \frac{r-2}{r-1}, \frac{r-3/2}{r-1/2}, \frac{r-1}{r}$. Let $n$ be sufficiently large, and let $G$ be $K_r$-semi-saturated on $n$ vertices with minimum number of edges and $\Delta(G)\le cn$. By Proposition~\ref{prop:T15-upper},
\begin{equation}\label{EUpBound}
e(G)\leq f(r,c) n +O(1).
\end{equation}

Set $M:=n^{0.4}$. Let $Y=\{v\in V(G):d(v)\ge M\}$ and $X=V(G)\setminus Y$.

For $u\in X$ write $d_Y(u)=|N(u)\cap Y|$ and $\delta_Y(X)=\min_{u\in X} d_Y(u)$. We will use $e(X,Y)$ to denote the number of edges between $X$ and $Y$ in $G$.
\vspace{0.5em}

\begin{lemma}\label{lem:T15-lb-bridge}
$|Y|=o(n)$ and $|X|=n-o(n)$.
\end{lemma}
\noindent\textbf{Proof. }Obviously $\sum_{v\in Y}d(v)\ge M|Y|$. On the other hand $\sum_{v\in Y}d(v)\le \sum_{v\in V(G)}d(v)=2e(G)\le 2f(r,c)n+O(1)$ by \eqref{EUpBound}. Therefore $|Y|\le 2f(r,c)n/M+O(1)=o(n)$, and $|X|=n-|Y|=n-o(n)$. \q
\vspace{0.5em}

\begin{lemma}\label{lem:T15-lb-delta-small}
$\delta_Y(X)\ge r-1$.
\end{lemma}
\noindent\textbf{Proof. }Suppose there is $u\in X$ such that $d_Y(u)\le r-2$. Count $3$-vertex paths with terminal vertex $u$:
\begin{equation*}
\begin{aligned}
& \bigl|\{(u,v,w): u,v,w\in V(G),\ uv,vw\in E(G),\ u\neq w\}\bigr|  \\
& \le (cn)\, d_Y(u) + (M-d_Y(u))(M-1) \\
& \le cn(r-2)+(M+2-r)(M-1).
\end{aligned}
\end{equation*}
For each $x\in V(G)\setminus N[u]$, we have $|N(u)\cap N(x)|\ge r-2$ since $G$ is $K_r$-semi-saturated which implies each common neighbour $v$ giving a $3$-vertex path $(u,v,x)$ counted on the left-hand side above. Since $|V(G)\setminus N[u]|=n-|N[u]|\ge n-M$, there are at least $(n-M)(r-2)$ such paths. So we have
\[
cn(r-2)+(M+2-r)(M-1)\ge (r-2)(n-M),
\]
equivalently $M^2-M+(c-1)(r-2)n+r-2\ge 0$. Substituting $M=n^{0.4}$ yields a contradiction for $n$ being sufficiently large. \q
\vspace{0.5em}

\begin{cor}\label{cor:T15-1}
For $\frac{r-1}{r}<c<1$, $A_{r-2}(c)=r-1$.
\end{cor}
\noindent\textbf{Proof. }By Lemmas~\ref{lem:T15-lb-bridge} and \ref{lem:T15-lb-delta-small}, we have $e(G) \ge e(X,Y)\ge (r-1)n - o(n) $.
By Proposition~\ref{prop:T15-upper} and Theorem~\ref{T14}, $A_{r-2}(c)=f(r,c)$ for $c \in (\frac{r-1}{r}, 1)$. \q
\vspace{0.5em}

If $\delta_Y(X)\ge r$, then $e(G)\ge e(X,Y)\ge rn-o(n)$. Note that $f(r,c)\le r$ and the equality holds only if $\frac{r}{r+2}<c<\frac{r-2}{r-1}$. By \eqref{EUpBound}, $\delta_Y(X)\le r-1$ when $c>\frac{r-2}{r-1}$.

By Lemma \ref{lem:T15-lb-delta-small}, we may assume $\delta_Y(X)=r-1$ in the following discussion.
Fix a vertex $u\in X$ with $d_Y(u)=r-1$. Let $B=N(u)\cap Y=\{b_1,\ldots,b_{r-1}\}$. Define
\[
W_1:=N(u)\cap X,\qquad W_2:=\bigl(N[W_1]\setminus N[u]\bigr)\cap X.
\]
Then $|W_1|+|W_2|+|B|+|\{u\}|=o(n)$ by $M=n^{0.4}$. Let $C$ be the set of the remaining $(n-o(n))$ vertices.

For every $p\in C$, we have $N(p)\cap N(u)\cap X=\emptyset$, and $|N(p)\cap B|\ge r-2$ since $pu\notin E(G)$ and $G$ is $K_r$-semi-saturated. Partition $C$ into
\[
A_0=\{x\in C: N(x)\cap B=B\},\qquad
A_i=\bigl\{x\in C: N(x)\cap B=B\setminus\{b_i\}\bigr\}\quad (1\le i\le r-1).
\]
Set $A=\bigcup_{i=1}^{r-1}A_i$. Then $C= A \cup A_0$ and
\begin{equation}\label{eq:Ai}
|A_i| \, \ge |C|-cn=(1-c-o(1))n
\end{equation}
for each~$ 1 \le i \le r-1$ (otherwise $\Delta > cn$). Then $0\le |A_0|\le |C|-\sum_{i=1}^{r-1}|A_i|\le n-(r-1)(1-c)n - o(n)$ which forces $c>\frac{r-2}{r-1}$. This shows that if $\frac{r}{r+2}<c<\frac{r-2}{r-1}$, we must have $\delta_Y(X)=r$. Therefore, we have following Corollary.
\vspace{0.5em}

\begin{cor}\label{cor:T15-4} For $\frac{r}{r+2}<c<\frac{r-2}{r-1}$, $A_{r-2}(c)=r$.
\end{cor}

By Corollaries \ref{cor:T15-1} and \ref{cor:T15-4},
 we only consider $ \frac{r-2}{r-1} < c < \frac{r-1}{r}$ and $c \neq \frac{r-3/2}{r-1/2}$. For $v_1,v_2\in C$, if $v_1\in A_0$ then $|N(v_1)\cap N(v_2)\cap B|=|N(v_2)\cap B|\ge r-2$; if $v_1\in A_i$, $v_2\in A_j$ with $i\neq j$, then
\begin{equation}\label{eq:Nv1v2B}
|N(v_1)\cap N(v_2)\cap B|=|B\setminus\{b_i,b_j\}| =r-3.
\end{equation}

Clearly
\begin{equation}\label{eq:CBE}
	e(C,B) = (r-1)|A_0|+(r-2)|A| =(r-1)n-|A|-o(n).
\end{equation}

Since $\delta_Y(X)=r-1$, we have $|N(a)\cap (Y\setminus B)| \ge 1$ for any $a\in A\cap X$. Suppose for each $a\in A\cap X$, $|N(a)\cap (Y\setminus B)|\ge 2$. Then $e(A\cap X,Y\setminus B)\ge 2(|A|-o(n))$ and by \eqref{eq:Ai} and \eqref{eq:CBE} we have
\begin{equation*}
    \begin{aligned}
        e(G) &\ge e(C,B)+e(A\cap X, Y\setminus B)\\
        &\ge (r-1)n-|A|+2|A|-o(n) \\
        &= (r-1)n+|A|-o(n)\\
        &\ge (r-1)n+(r-1)(1-c)n-o(n).
    \end{aligned}
\end{equation*}
Since
\begin{equation*}
	(r-1)+(r-1)(1-c) > (3r-4)-(2r-2)c > (2r-3)-(r-2)c
\end{equation*}
holds for $c \in (\frac{r-2}{r-1}, \frac{r-1}{r})$, we have $e(G) > f(r,c) n +o(n)$ for sufficiently large $n$, a contradiction with \eqref{EUpBound}.

So there exists  $a^*\in A \cap X$, say $a^* \in A_1\cap X$, such that $|N(a^*)\cap (Y\setminus B)|=1$, say $u^* \in N(a^*)\cap (Y\setminus B)$.
For any $w\in (A\setminus A_1) \cap X$, since $G$ is $K_r$-semi-saturated, either $a^*w\in E(G)$ or $|N(a^*) \cap N(w) \cap (V(G) \setminus B) | \, \geq 1$ by \eqref{eq:Nv1v2B}.
Partition $(A\setminus A_1) \cap X$ into the following three sets:
\begin{equation*}
	\begin{aligned}
		S_1 &= \{w\in (A\setminus A_1) \cap X: a^*w\in E(G)\}, \\
		S_2 &= \{w\in (A\setminus A_1) \cap X: a^*w\notin E(G), N(a^*) \cap N(w) \cap (Y \setminus B)  \, = \emptyset \}, \\
		S_3 &= \{w\in (A\setminus A_1) \cap X: a^*w\notin E(G), N(a^*) \cap N(w) \cap (Y \setminus B) = \{u^*\}\}.
	\end{aligned}
\end{equation*}

Since $a^*\in X$, we have $d(a^*)<M$ which implies $|S_1|\, \le M = o(n)$.
If $w \in S_2$, then $(N(a^*)\cap N(w))\setminus B\subseteq X$ which implies $|S_2|\, \le M^2 = o(n)$. Then
\begin{equation*}
	|S_3|\, = | (A\setminus A_1) \cap X| - |S_1| - |S_2| \ge (r-2)(1-c)n - o(n),
\end{equation*}
which implies $d(u^*)\, \ge (r-2)(1-c)n - o(n)$ and $u^*$ is completely joined to all but $o(n)$ vertices in $(A \setminus A_1) \cap X$.
\vspace{0.3em}

\begin{claim}\label{clm:T15-lb-residual-2}
	For any $w' \in (A \setminus A_1) \cap X \cap N(u^*)$, $w'$ has at least two neighbours in $Y\setminus B$.
\end{claim}
\noindent\textbf{Proof of Claim \ref{clm:T15-lb-residual-2}. }Suppose there is $w' \in (A \setminus A_1) \cap X \cap N(u^*)$ such that $|N(w')\cap (Y\setminus B)|=1$. By replacing $a^*$ with $w'$ and  the same argument as above, we have $u^*$ is adjacent to all but $o(n)$ vertices in $ A\cap X$, which implies $d(u^*)\, \ge (r-1)(1-c)n - o(n) > cn\ge \Delta$ for $c < \frac{r-1}{r}$, a contradiction. $\hfill\square$
\vspace{0.3em}

By Claim \ref{clm:T15-lb-residual-2}, all but $o(n)$ vertices in $(A \setminus A_1) \cap X$ have at least two neighbours in $Y\setminus B$, and all vertices in $A_1 \cap X$ have at least one neighbour in $Y\setminus B$. Hence
\begin{equation}\label{eq:G-lower-bound-2}
    \begin{aligned}
        e(G) &\ge e(C,B)+e(A\cap X,Y\setminus B)\\
        &\ge (r-1)n-|A|+|A_1|+2\sum_{i=2}^{r-1}|A_i|-o(n)\\
        &\ge (r-1)n+\sum_{i=2}^{r-1}|A_i|-o(n)\\
        &\ge (r-1)n+(r-2)(1-c)n-o(n)\\
		& = (2r-3-(r-2)c)n-o(n).
    \end{aligned}
\end{equation}

\begin{cor}\label{cor:T15-2}
	For $\frac{r-3/2}{r-1/2}<c<\frac{r-1}{r}$, $A_{r-2}(c)=2r-3-(r-2)c$.
\end{cor}
\noindent\textbf{Proof. }By \eqref{eq:G-lower-bound-2} we have $e(G)\ge (2r-3-(r-2)c)n-o(n)$. By Proposition~\ref{prop:T15-upper} and Theorem~\ref{T14}, $A_{r-2}(c)=f(r,c)$ for $c \in (\frac{r-3/2}{r-1/2}, \frac{r-1}{r})$. \q
\vspace{0.5em}

The final remaining case is $\frac{r-2}{r-1} < c < \frac{r-3/2}{r-1/2}$. Since $\Delta \leq cn$, at most $cn-(r-2)(1-c)n+o(n)$ vertices of $A_1\cap X$ are adjacent to~$u^*$. Since all vertices in $A_1 \cap X$ have at least one neighbour in $Y\setminus B$, we consider two cases.

\medskip\noindent\textbf{Case 1.} There exists $x^* \in (A_1\cap X)\setminus N(u^*)$ such that $|N(x^*)\cap (Y\setminus B)|=1$.

Let  $y^*\in N(x^*)\cap (Y\setminus B)$. By replacing $u^*$ with $y^*$ and  the same argument as above,
we have $y^*$ is completely joined to all but $o(n)$ vertices in $(A \setminus A_1) \cap X$.
Let $A_{1}(\xi) = A_1 \cap X \cap N(\xi)$ for $\xi \in \{u^*,y^*\}$ and $A^*= (A \setminus A_1) \cap X \cap N(u^*) \cap N(y^*) $. Then $|A_1(\xi)|\, \le cn-(r-2)(1-c)n+o(n)$ and $|A^*| \, \ge (r-2)(1-c)n - o(n)$.

By Lemma~\ref{lem:T15-lb-bridge} and \eqref{eq:Ai}, $|A_1 \cap X|\ge (1-c)n-o(n)$.
As $c < \frac{r-3/2}{r-1/2}$, we have $2c-2(r-2)(1-c) < 1-c $ which implies there exists $ a' \in (A_1 \cap X) \setminus (N(u^*) \cup N(y^*))$. Let $A'=(N[a']\cup N(N(a')))\cap X$. Then $|A'|\le M+M^2=o(n)$.
For any $z\in A^*\setminus A'$,  $N(z) \cap (Y \setminus B) \supsetneq \{u^*,y^*\}$;  otherwise $a'z \not \in E(G)$ and $|N(a' ) \cap N(z)| =r-3$, a contradiction with $G$ being $K_r$-semi-saturated.
Therefore all but $o(n)$ vertices in $(A \setminus A_1) \cap X $ have at least three neighbours in $Y\setminus B$. Hence
\begin{equation*}
    \begin{aligned}
        e(G) &\ge e(C,B)+e(A\cap X,Y\setminus B)\\
        &\ge (r-1)n-|A|+|A_1|+3\sum_{i=2}^{r-1}|A_i|-o(n)\\
        &\ge (r-1)n+2\sum_{i=2}^{r-1}|A_i|-o(n)\\
        &\ge (r-1)n+2(r-2)(1-c)n-o(n)\\
        &> rn > f(r,c) n +o(n)
    \end{aligned}
\end{equation*}
for $ c \in (\frac{r-2}{r-1}, \frac{r-3/2}{r-1/2})$ and $n$ sufficiently large, a contradiction with \eqref{EUpBound}.
\vspace{0.2em}

\medskip\noindent\textbf{Case 2.} For every $x\in (A_1\cap X)\setminus N(u^*)$,  $|N(x)\cap (Y\setminus B)|\ge 2$.

Let $A_{1,1} = A_1 \cap X \cap N(u^*)$ and $A_{1,2} = (A_1 \cap X) \setminus A_{1,1}$. Then $|A_{1,1}| \le cn-(r-2)(1-c)n+o(n)$. Since every $x\in A_{1,2}$ has at least two neighbours in $Y\setminus B$, we have
\begin{equation}\label{eq:G-lower-bound-3}
    \begin{aligned}
        e(G) &\ge e(C,B)+e(A\cap X,Y\setminus B)\\
        &\ge (r-1)n-|A|+|A_{1,1}|+2(|A|-|A_{1,1}|)-o(n)\\
        &\ge (r-1)n+|A|-|A_{1,1}|-o(n)\\
        &\ge (r-1)n+(r-1)(1-c)n-\bigl(cn-(r-2)(1-c)n\bigr)-o(n)\\
        &\ge \bigl((3r-4)-(2r-2)c\bigr)n-o(n).
    \end{aligned}
\end{equation}

\begin{cor}\label{cor:T15-3}
	For $\frac{r-2}{r-1}<c<\frac{r-3/2}{r-1/2}$, $A_{r-2}(c)=(3r-4)-(2r-2)c$.
\end{cor}
\noindent\textbf{Proof. }By \eqref{eq:G-lower-bound-3} we have $e(G)\ge ((3r-4)-(2r-2)c)n-o(n)$. By Proposition~\ref{prop:T15-upper} and Theorem~\ref{T14}, $A_{r-2}(c)=f(r,c)$ for $c \in (\frac{r-2}{r-1}, \frac{r-3/2}{r-1/2})$. \q
\vspace{0.5em}

\noindent\textbf{Proof of Theorem~\ref{T15}. }By Corollaries~\ref{cor:T15-1}, \ref{cor:T15-4}, \ref{cor:T15-2},~and \ref{cor:T15-3}, we are done. \q
\vspace{0.5em}

Subsections~\ref{subsec:T15-upper} and~\ref{subsec:T15-lb} together prove Theorem~\ref{T15}, hence $A_{r-2}(c)=f(r,c)$ on each open $c$-interval stated there.
It follows that, on every such interval, the linear program~\eqref{eq:1.3} attains the value $A_{r-2}(c)$ among all $(r-2)$-intersecting hypergraphs $\mathcal{H}$ with $1/\nu^*(\mathcal{H})\le c$.
The incidence matrices exhibiting the corresponding feasible solutions are recorded in Appendix~\ref{app:T15-matrices}.
We emphasize that Proposition~\ref{prop:T15-upper} is proved entirely by the graph constructions in Appendix~\ref{app:T15-constructions}; the hypergraph matrices in Appendix~\ref{app:T15-matrices} are only a visualization of the neighborhood patterns occurring in Theorem~\ref{T15}, and do not themselves constitute a proof of the proposition.

\section{Specific maximum degree restriction}

In this section, we consider the specific maximum degree restriction and its applications. We prove Theorem~\ref{T16} first.
\vspace{0.5em}

\noindent \textbf{Proof of Theorem~\ref{T16}. }
Let $$T:=(r-1)n-\binom{r}{2}=(r-1)(n-r)+\binom{r}{2}.$$
Let $B=\{b_1,\ldots,b_r\}$ and partition $V(G)\setminus B$ into $r$ parts $P_1,\ldots,P_r$ with sizes differing by at most~$1$.
Let $G$ be a graph on $n$ vertices defined as follows: $G[B]\cong K_r$; for each $1\le i\le r$, join every vertex of $P_i$ to every vertex of $B\setminus\{b_i\}$;
Then $\Delta(G)=n-1-\lfloor\frac{n-r}{r}\rfloor$ and $G$ is $K_r$-semi-saturated.
Counting edges gives
\[
e(G)=\binom{r}{2}+\sum_{i=1}^r |P_i|(r-1)=(r-1)(n-r)+\binom{r}{2}=T,
\]
so $ssat^{\Delta}(n,K_r)\le T$ for any $\Delta\ge n-1-\lfloor\frac{n-r}{r}\rfloor$.

When $(r,\Delta)=(4,n-2)$, let $G^*=(K_2 \cup K_1) \vee ( K_{n-5}^c \cup K_2)$, where $K_{n-5}^c$ is the complement of $K_{n-5}$. Then  $G^*$ is $K_4$-semi-saturated with $\Delta(G^*)=n-2$ and  $e(G^*)=T-1=3n-7$. So $ssat^{n-2}(n,K_4)\le T-1$.

For the lower bound, let $G$ be a $K_r$-semi-saturated graph on $n$ vertices with $\Delta(G)\le n-2$.
We prove $e(G)\ge T$ except for $(r,\Delta)=(4,n-2)$.

For $u\in V(G)$ and $A\subseteq V(G)$ let $N_A(u):=N(u)\cap A$ and $d_A(u):=|N_A(u)|$.
Since $G$ is $K_r$-semi-saturated and $\Delta(G)\le n-2$, we have $\delta(G)\ge r-1$.
We distinguish three cases according to $\delta(G)$.

\medskip
\noindent\textbf{Case 1.} $\delta(G)\ge r+1$.
If $\delta(G)\ge 2r$, then $e(G)\ge n\delta(G)/2\ge rn>T$ and we are done. So we assume $\delta(G)=o(n)$.
Let $v\in V(G)$ with $d(v)=\delta(G)$.
For every $u\in V(G)\setminus N[v]$, the graph $G+uv$ contains a copy of $K_r$, so $d_{N(v)}(u)\ge r-2$ and $u$ has at least $\delta(G)-d_{N(v)}(u)$ neighbours in $V(G)\setminus N[v]$.
Summing over $u\in V(G)\setminus N[v]$ and adding the edges between $N(v)$ and $V(G)\setminus N[v]$ gives
\begin{equation*}
    \begin{aligned}
e(G)&\ge \sum_{u\in V(G)\setminus N[v]}d_{N(v)}(u)+\frac{1}{2}\sum_{u\in V(G)\setminus N[v]}(\delta(G)-d_{N(v)}(u))\\
&=\sum_{u\in V(G)\setminus N[v]}\left(d_{N(v)}(u)+\frac{1}{2}(\delta(G)-d_{N(v)}(u))\right)\\
&\ge \frac{(\delta(G)+r-2)(n-\delta(G)-1)}{2}
\ge \Bigl(r-\tfrac{1}{2}\Bigr)(n-2r)>T
\end{aligned}
\end{equation*}
for sufficiently large $n$.

\medskip
\noindent\textbf{Case 2.} $\delta(G)=r-1$.
Let $v\in V(G)$ with $d(v)=r-1$.
Pick an arbitrary $x\in V(G)\setminus N[v]$.
Since $G$ is $K_r$-semi-saturated, we have $G[N(x)\cap N(v)]\supseteq K_{r-2}$; let $u$ denote the unique vertex of $N(v)$ not in this $K_{r-2}$. So $G[N(v)\setminus\{u\}]\cong K_{r-2}$.
Write $W:=V(G)\setminus N[v]$ and $t:=e\bigl(\{u\},N(v)\setminus\{u\}\bigr)$.
We distinguish three subcases according to~$t$.

\medskip
\noindent\textbf{Subcase 2.1.} $t=0$.
In this case, $N(v)\setminus\{u\}\subseteq N(w)$ for every $w\in W$ which implies $W$ and $N(v)\setminus\{u\}$ are completely joined.

Set $S:=N(u)\cap W$ and $S_1:=W\setminus S$. For $u_1\in N(v)\setminus \{u\}$, $G+uu_1$ contains a new copy of $K_r$, which implies $G[S]\supseteq K_{r-2}$. For any $z\in S_1$(possibly $S_1=\emptyset$), $G+uz$ contains a new copy of $K_r$, which implies $d_S(z)\ge r-2$. Therefore,
\[
\begin{aligned}
e(G) &\ge \binom{r-2}{2}+(r-1)+(r-2)|W|+\binom{r-2}{2}+|S|+e(S,S_1)\\
&\ge (r-1)+(n-r)(r-2)+2\binom{r-2}{2}+|S|+(r-2)|S_1|\\
&\ge (n-r+1)(r-1)+(r-2)(r-3).
\end{aligned}
\]
Since $(r-2)(r-3)\ge (r-1)^2-\binom{r}{2}$ if and only if $r\ge 5$, this gives $e(G)\ge T$ for $r\ge 5$. When $r=4$, we obtain $e(G)\ge 3n-7=T-1$.
Equality at $T-1$ is attained when $r=4$ by the  extremal graph:
$G^*=(K_2 \cup K_1) \vee ( K_{n-5}^c \cup K_2)$ where $\Delta(G^*)=n-2$.

\medskip
\noindent\textbf{Subcase 2.2.} $1\le t\le r-3$.
Pick $u_1\in N(u)\cap (N(v)\setminus\{u\})$.
For every $w\in W$, since $G+wv$ contains a new copy of $K_r$, we have $w\in N(u_1)$, which implies $d(u_1)=n-1$, a contradiction to $\Delta(G)\le n-2$.

\medskip
\noindent\textbf{Subcase 2.3.} $t=r-2$, i.e.\ $G[N(v)]\cong K_{r-1}$.
Set $N(v)=\{b_1,\ldots,b_{r-1}\}$.
Partition $W$ into
\[
\begin{aligned}
\widetilde{A}_0&:=\{x\in W: N(x)\cap N(v)=N(v)\},\\
\widetilde{A}_i&:=\{x\in W: N(x)\cap N(v)=N(v)\setminus\{b_i\}\}\qquad (1\le i\le r-1),
\end{aligned}
\]
and write $\widetilde{A}:=\bigcup_{i=1}^{r-1}\widetilde{A}_i$. Since $\Delta\le n-2$, $\widetilde{A}_i\not=\emptyset$ for $1\le i\le r-1$.
For $a_i\in \widetilde{A}_i$ and $a_j\in \widetilde{A}_j$ with $1\le i\neq j\le r-1$, if $a_ia_j\notin E(G)$, then $|N(a_i)\cap N(a_j)|\ge r-2$ which implies $N(a_i)\cap N(a_j)\cap  (V(G)\setminus N[v])\not=\emptyset$.
Let $G'=G[V(G)\setminus N[v]]$. Then there is a component~$C$ of~$G'$ such that $\widetilde{A}\subseteq V(C)$.

\smallskip
First, suppose that $V(C)\cap \widetilde{A}_0\not=\emptyset$ or $C$ has a cycle.
Then
\[
e(G)\ge |\widetilde{A}_0|(r-1)+\Bigl[|\widetilde{A}|(r-2)+\binom{r-1}{2}+|\widetilde{A}|-1+1\Bigr]
= (n-r)(r-1)+\binom{r}{2}=T.
\]

\smallskip
Now suppose that $C$ is a tree and $V(C)\cap \widetilde{A}_0=\emptyset$. Let
 $x^* \in \widetilde{A}_i$, say $i=1$, be a leaf of $C$ and $y^* \in \widetilde{A}_j$ its neighbour in $C$. Then $y^*x\in E(G)$ for  all $x\in \widetilde{A} \setminus \widetilde{A}_1$; otherwise $G+x^*x$ would not create a new copy of $K_r$, a contradiction.
Since $r-1 \geq 3$, there are $x_1 \in \widetilde{A}_{i_1}$ and $x_2 \in \widetilde{A}_{i_2} $ such that $i_1,i_2,j$ are all different and $y^*\in N(x_1) \cap N(x_2)$. Then $|N(x_1) \cap N(x_2) \cap N(y^*) \cap N(v) |= r-4$.  Since $C$ is a tree, we have that $x_1x_2\notin E(G)$.
As $G+x_1x_2$ contains a copy of $K_r$, there is $w^* \in (\widetilde{A}_0 \cup \widetilde{A})\setminus\{y^*\}$ such that $w^*\in N(x_1)\cap N(x_2)$
 which means either $C$ contains a vertex of $\widetilde{A}_0$ or $C$ has a cycle, a contradiction.

\medskip
\noindent\textbf{Case 3.} $\delta(G)=r$.
Let $v\in V(G)$ with $d(v)=r$. Since $G$ is $K_r$-semi-saturated, $|N(x)\cap N(v)|\ge r-2$ for any $x\in V(G)\setminus N[v]$.
Partition $V(G)\setminus N[v]$ into
\[
B:=\{x\in V(G): d_{N(v)}(x) \geq r-1\},\qquad
C:=\{x\in V(G): d_{N(v)}(x)=r-2\},
\]
and set $b:=|B|$ and $c:=|C|$. Then $b+c=n-r-1$.

\begin{claim}\label{clm:T16-b} We can assume
$b\le 4r-8$.
\end{claim}
\noindent\textbf{Proof of Claim \ref{clm:T16-b}. }
If $b > 4r-8$, then
\begin{align*}
e(G) &\ge r+b\Bigl(r-\frac{1}{2}\Bigr)+c(r-1)+\binom{r-2}{2}\\
&=(n-r)(r-1)+1+\frac{b}{2}+\binom{r-2}{2}\\
&>(n-r)(r-1)+1+2r-4+\binom{r-2}{2}\\
&=(n-r)(r-1)+\binom{r}{2}=T
\end{align*}
and we are done.
$\hfill\square$

By the same argument as that of Claim~\ref{clm:T16-b}, we can assume $\Delta\bigl(G[B\cup C]\bigr)\le 5r$, otherwise $e(G)>T$ and we are done.
Since $b\le 4r-8$, we have $c\ge n-5r$.
We consider two subcases.

\medskip
\noindent\textbf{Subcase 3.1.} All vertices of $C$ have the same neighbourhood in~$N(v)$.
Fix $c_0\in C$ and write $S:=N(c_0)\cap N(v)$. Then $N(c)\cap N(v)=S$ for every $c\in C$ and $N(v)=S\cup\{u_1,u_2\}$, where  $u_1,u_2\in N(v)\setminus S$.
Since $\Delta(G)\le n-2$, there exists
\[
w\in B\cup\{u_1,u_2\}
\quad\text{with}\quad
N(w)\not\supseteq S.
\]
Then $|N(w)\cap N(x)\cap (V(G)\setminus N(v))|\ge 1$ for any $x\in C\setminus N[w]$ since $G$ is $K_r$-semi-saturated.
Let $M:=N_{B\cup C}(w)$. Then $|M| \, \le 5r$  and $|N_C(M)|\,\ge |C|-|M|=c-|M| \geq n-10r$.
Hence
\[
e(G)\ge (n-r-|M|)(r-1)+\frac{n-10r}{2}>(n-r)(r-1)+\binom{r}{2}=T
\]
for sufficiently large $n$.

\medskip
\noindent\textbf{Subcase 3.2.} There exist $y,y'\in C$ with $N(y)\cap N(v)\neq N(y')\cap N(v)$.
Since  $|N(x)\cap N(v)|=r-2$ for any $x\in C$  and $|C|\, \ge n-5r$, by the pigeonhole principle, there is a subset $C_1\subseteq C$ with
\[
|C_1| \, \ge \frac{n-5r}{\binom{r}{r-2}}=\frac{2(n-5r)}{r(r-1)}
\]
such that $N(y)\cap N(v)= S$ for every $y\in C_1$.
Pick $c_2\in C\setminus C_1$. Then
$|N(c_2)\cap N(y)\cap (V(G)\setminus N(v))|\ge 1$ for any $y\in C_1\setminus N[c_2]$ since $G$ is $K_r$-semi-saturated.
Let $M:=N_{B\cup C}(c_2)$. Then $|M|\,\le 5r$ and $|N_{C_1}(M)| \, \ge |C_1| - |M| \, \ge \dfrac{2(n-5r)}{r(r-1)} - 5r$.
Hence
\[
e(G)\ge (n-r-|M|)(r-1)+\frac{n-5r}{r(r-1)}-\frac{5r}{2}>(n-r)(r-1)+\binom{r}{2}=T
\]
for sufficiently large $n$. This completes the proof of Theorem~\ref{T16}. \q
\vspace{0.5em}

We now prove Theorem~\ref{T18}.
Note that every ($K_s\vee F$)-saturated graph is $K_{s+t}$-semi-saturated. So the bounds from Section~2 and Theorem~\ref{T16} for $\Delta=n-2$ apply with $r=s+t$.
When $r=3$ we apply Theorem~\ref{T153} instead of Theorem~\ref{T16}; for $r\ge 4$ the lower bound of Theorem~\ref{T16} already enters the proof. We first have the following lemma.
\vspace{0.5em}

\begin{lemma}\label{clm:T18-conical}
	Let $F'$ be a graph. For any positive integers $p \leq q$ and $n$ sufficiently large, if $G$ is a graph on $n$ vertices and $G$ contains $p$ conical vertices $u_1,u_2,\ldots,u_p$, then $G$ is $(K_q \vee F')$-(semi)-saturated if and only if $H :=G-\{u_1,u_2,\ldots,u_p\}$ is $(K_{q-p} \vee F')$-(semi)-saturated.
\end{lemma}
\noindent\textbf{Proof. }Clearly $G$ is $(K_q \vee F')$-free if and only if $H$ is $(K_{q-p}\vee F')$-free. The implication ``$\Leftarrow$'' is immediate, so we only prove ``$\Rightarrow$''. Let $U:=\{u_1,u_2,\ldots,u_p\}$.

Let $xy$ be a nonedge in $H$ (and thus a nonedge in $G$).
Since $G$ is $(K_q \vee F')$-(semi)-saturated,  $G+xy$ contains a copy of $K_q\vee F'$, say $T$.
Fix a decomposition $V(T)=A\sqcup B$ with $|A|=q$ and $T[A]\cong K_q$, $T[B]\cong F'$. Let $V(T) \cap U = U^*$ with $|U^*|=p' \leq p$.

\smallskip
\noindent\textbf{Case 1:} $B\cap U^*=\emptyset$.
Then $B\subseteq V(H+xy)$, and also $A\setminus U^* \subseteq V(H+xy)$.
The induced subgraph $T-U^*$ has clique side $A\setminus U ^*$ of size $q-p' \geq q-p$, $F'$-side $B$, and still has all edges between them inherited from~$T$.
Thus $T-U^*\subseteq H+xy$ is a copy of $K_{q-p'}\vee F'$ which implies that $H+xy$ creates a new copy of $K_{q-p} \vee F'$.

\smallskip
\noindent\textbf{Case 2:} $B\cap U^*\not=\emptyset$.
Fix $u\in U^*\cap B$ and choose any $w\in A \setminus U^*$.
Since $u$ is conical, it is adjacent to all vertices of $(A\setminus\{w\})\cup(B\setminus\{u\})$.
Also $w$ is adjacent in $T$ to every vertex of $(A\setminus\{w\})\cup(B\setminus\{u\})$ because $T[A]$ is a clique and $w$ is joined completely to~$B$.
Thus $u$ and $w$ coincide on their neighbourhoods restricted to $V(T)\setminus\{u,w\}$.
Therefore we can exchange the roles of $u$ and $w$ inside the abstract isomorphism type $K_q\vee F'$: there exists another copy $T'\cong K_q\vee F'$ contained in $G+xy$, on the same vertex set as $T$ up to relabelling the distinguished $q$-clique, such that $u$ lies in the clique side of $T'$ rather than in its $F'$-side.
Replacing $T$ by $T'$ and repeating if necessary, we may assume $B\cap U^*=\emptyset$, and then Case~1 applies. \q

Now we are going to prove Theorem~\ref{T18}.

\vspace{0.5em}

\noindent\textbf{Proof of Theorem~\ref{T18}. }Write $H_s:=K_s\vee F$ and $r:=s+t$. Then every edge of $H_s$ belongs to a $K_r$. Since $s \geq 1$ and $t\ge 2$, we have $r\ge 3$. Let $G$ be a $H_s$-saturated graph on $n$ vertices with minimum number of edges. Then $G$ is $K_r$--semi-saturated. By \eqref{eq:1.2}, we have
\begin{equation}\label{eq:4.1}
	e(G)= \min \{ \, sat^{ n-2}(n,H_s) \, , \; sat^{n-1}_*(n,H_s) \, \} \, .
\end{equation}
Moreover
\begin{equation}\label{eq:T18-ssat-mono}
	\mathrm{sat}^{\Delta}(n,H_s) \ge \mathrm{ssat}^{\Delta}(n,H_s) \ge \mathrm{ssat}^{\Delta}(n,K_r)
\end{equation}
holds for all $\Delta$. Recall that
\begin{equation}\label{eq:T18-ssat-mono-2}
	sat(n,K_s \vee F')  \leq  s(n-s)+sat(n-s,F')+ \binom{s}{2}
\end{equation}
holds for all $F'$ and $n > |V(F')| \, +s-1$.

\vspace{0.5em}

By Theorems~\ref{T153} and~\ref{T16}, and \eqref {eq:T18-ssat-mono} and \eqref {eq:T18-ssat-mono-2} we have
	\begin{equation}\label{eq:sat-lower-bound}
		\begin{aligned}
			sat^{n-2}(n,H_s) & \geq ssat^{n-2}(n,H_s) \geq ssat^{n-2}(n,K_r) \\
			& \geq  (r-1)n- \binom{r}{2}-2= (t+s-1)n- \binom{t+s}{2} \\
			& = s(n-s) + (t-1)(n-s) - \binom{t}{2} -2 + \binom{s}{2} \\
			& > s(n-s) + sat(n-s,F) + \binom{s}{2} \geq sat(n,H_s)\,.
		\end{aligned}
	\end{equation}
	Combined with  \eqref{eq:4.1} we have $sat(n,H_s) = sat_*^{n-1}(n,H_s) $ which implies that $G$ contains a conical vertex $u_1$.
	Let $G_1=G-u_1$. Then $G_1$ is $H_{s-1}$-saturated by Lemma \ref{clm:T18-conical}.
	Repeating this process on $G_1$ as in  \eqref{eq:sat-lower-bound} by taking $n=n-1$ and $s=s-1$, we have $G_1$ contains a conical vertex $u_2$, and $G_2=G_1-u_2$ is $H_{s-2}$-saturated.
	Continuing this process $s$ times, we obtain a sequence $G =G_0\supseteq G_1 \supseteq G_2 \supseteq \cdots \supseteq G_s$ and conical vertices $u_1,u_2,\ldots,u_s$ in $G$. By Lemma \ref{clm:T18-conical} we have $e(G_s) \geq sat(n-s,F)$. Hence
	\begin{equation*}
	\begin{aligned}
		 sat(n,H_s) = e(G) = s(n-s)+ \binom{s}{2} + e(G_s) \geq s(n-s)+ sat(n-s,F)+ \binom{s}{2} \, .
	\end{aligned}
	\end{equation*}
Combined with  \eqref{eq:T18-ssat-mono-2} we have $sat(n,K_s \vee F) = s(n-s)+ sat(n-s,F)+ \binom{s}{2}  $ and we are done. \q
\vspace{0.5em}

\noindent\textbf{Remark. }We give an example to show that for any integer $t\geq 2$, there exists $F$ satisfying the condition in Theorem~\ref{T18}. Let $F = K_t \cup M$, where every edge of the graph $M$ belongs to a $K_t$. For sufficiently large $n$ let $G=K_{t-2} \vee (K_{|M|+1} \cup K_{n-t-|M|+1}^c)$. Not difficult to check that $G$ is $F$-saturated and
\begin{equation*}
	sat(n,F) \leq e(G) = (t-2)(n-t+2) + \binom{t-2}{2} + \binom{|M|+1}{2}
\end{equation*}
which implies that $F$ satisfies the condition in Theorem~\ref{T18}.

\section*{Acknowledgement}
This paper is supported by Beijing Natural Science Foundation (No.~QG26001); and by the National Natural Science Foundation of China (No.~12571372, 12401445); and by the Talent Fund of Beijing Jiaotong University (no. 2024-003).

\appendix

\section{Explicit graph constructions for Proposition~\ref{prop:T15-upper}}\label{app:T15-constructions}

We give explicit graph-theoretic realizations of Constructions~\textbf{(I)}--\textbf{(IV)} from Proposition~\ref{prop:T15-upper}.
In each case we specify vertices and edges, verify $\Delta(G)\le cn$ for $c$ in the stated interval, check $K_r$-semi-saturation, and count $e(G)$.
We simply omit rounding when assigning part sizes; since $n$ is large, this does not affect the limiting quantities.

\medskip
\noindent\textbf{Construction (I).}
Let $B=\{b_1,\ldots,b_r\}$ and partition $V(G)\setminus B$ into $A_1,\ldots,A_r$ as evenly as possible. Let $B$ induce a clique, and for each $i$ join every vertex of $A_i$ to every vertex of $B\setminus\{b_i\}$ (no edges between distinct parts $A_i,A_j$).
One can easily check that $\Delta(G)\le cn$ when $c>\frac{r-1}{r}$ and $n$ is sufficiently large. The graph is $K_r$-semi-saturated (every missing edge meets $B$ in a way that yields $r-2$ common neighbours in $B$) and
\[
e(G)=\binom{r}{2}+\sum_{i=1}^r |A_i|(r-1)=(r-1)(n-r)+\binom{r}{2}=(r-1)n+O(1).
\]

\medskip
\noindent\textbf{Construction (II).}
Let $B:=\{b_1,\ldots,b_{r-1},b_r,b_{r+1}\}$ induce a clique. Partition $V(G)\setminus B$ into disjoint sets $A_0,A_{1,1},A_{1,2},A_2,\ldots,A_{r-1}$, with no edges between distinct parts among these sets.
Join each vertex of $A_0$ to each of $b_1,\ldots,b_{r-1}$ (and to no vertex of $\{b_r,b_{r+1}\}$).
For each $j\in\{2,\ldots,r-1\}$, join each vertex of $A_j$ to each vertex of $\{b_1,\ldots,b_{r-1}\}\setminus\{b_j\}$.
Join each vertex of $A_{1,1}$ and each vertex of $A_{1,2}$ to each vertex of $\{b_2,\ldots,b_{r-1}\}$.
Join each vertex of $A_{1,1}\cup \left( \bigcup_{j=2}^{r-1}A_j \right)$ to $b_r$, and each vertex of $A_{1,2}\cup \left( \bigcup_{j=2}^{r-1}A_j \right)$ to $b_{r+1}$.
Set $|A_0|=\bigl(1-(r-1)(1-c)\bigr)(n-r-1)-r^2$,  $|A_{1,1}| = |A_{1,2}| = \tfrac{1}{2}\bigl(1-c \bigr)(n-r-1)+ r$, $|A_j|=(1-c)(n-r-1)+r$ for $j\ge 2$, feasible for $\frac{r-3/2}{r-1/2}<c<\frac{r-1}{r}$.
Let $S:=\sum_{j=2}^{r-1}|A_j|= (r-2)(1-c)(n-r-1)+(r-2)r$ and $A_1=A_{1,1}\cup A_{1,2}$. Since  $\bigl(r-\tfrac{3}{2}\bigr)(1-c) < c$ for $\frac{r-3/2}{r-1/2}<c<\frac{r-1}{r}$, one can check that $\Delta(G)\le cn$ for $n$ large.
Clearly $G$ is $K_r$-semi-saturated. Counting edges between $V(G)\setminus B$ and $B$ and adding $\binom{r+1}{2}$ edges inside $B$ gives
\[
e(G)=\bigl((r-1)+(r-2)(1-c)\bigr)n+O(1).
\]

\medskip
\noindent\textbf{Construction (III).}
Let $B:=\{b_1,\ldots,b_{r-1},b_r,b_{r+1},b_{r+2}\}$ induce a clique. Partition $V(G)\setminus B$ into disjoint sets $A_0,A_{1,1},A_{1,2},A_2,\ldots,A_{r-1}$, with no edges between distinct parts. Join each vertex of $A_0$ to each of $b_1,\ldots,b_{r-1}$ only.
For each $j\in\{2,\ldots,r-1\}$, join each vertex of $A_j$ to each vertex of $\{b_1,\ldots,b_{r-1}\}\setminus\{b_j\}$. Join each vertex of $A_{1,1}$ and each vertex of $A_{1,2}$ to each vertex of $\{b_2,\ldots,b_{r-1}\}$.
Join each vertex of $A_{1,1}\cup \left( \bigcup_{j=2}^{r-1}A_j \right)$ to $b_r$. Join each vertex of $A_{1,2}\cup A_2$ to $b_{r+1}$. Join each vertex of $A_{1,2}\cup \left( \bigcup_{j=3}^{r-1}A_j \right)$ to $b_{r+2}$.
Set $|A_0|=\bigl(1-(r-1)(1-c) \bigr)(n-r-2) -(r-1)(r+1)$, $|A_{1,1}|=\bigl(c-(r-2)(1-c)\bigr)(n-r-2)  -r(r+1)$, $|A_{1,2}|=\bigl((r-1)(1-c)-c \bigr)(n-r-2)+(r+1)^2$, $|A_j|=(1-c)(n-r-2)+r+1$ for $j\ge 2$, feasible for $\frac{r-2}{r-1}<c<\frac{r-3/2}{r-1/2}$.
Let $S:=\sum_{j=2}^{r-1}|A_j|= (r-2)(1-c)(n-r-2)+(r-2)(r+1)$, $T:=\sum_{j=3}^{r-1}|A_j|=S-|A_2|=(r-3)(1-c)(n-r-2)+(r-3)(r+1)$, and $A_1=A_{1,1}\cup A_{1,2}$.
One can check that $\Delta(G)\le cn$ for $n$ large.
Clearly $G$ is $K_r$-semi-saturated. Counting edges between $V(G)\setminus B$ and $B$ and adding $\binom{r+2}{2}$ edges inside $B$ gives
\[
e(G)=\bigl((3r-4)-(2r-2)c\bigr)n+O(1).
\]

\medskip
\noindent\textbf{Construction (IV).}
Let $B=\{b_1,\ldots,b_{r+2}\}$ induce a clique and partition $V(G)\setminus B$ into $A_1,\ldots,A_{r+2}$ as evenly as possible. For each $1\le i\le r+2$, join every vertex of $A_i$ to every vertex of $B\setminus\{b_i,b_{i+1}\}$, with indices read modulo $r+2$ (so $b_{r+3}=b_1$).
Each vertex of $\bigcup_{i=1}^{r+2} A_i$ has exactly $r$ neighbours in $B$. One can check that $\Delta(G)\le cn$ when $c>\frac{r}{r+2}$ and $n$ is large.
Every missing edge between vertices in distinct parts $A_i,A_j$ has enough common neighbours in $B$ to complete a $K_r$, so $G$ is $K_r$-semi-saturated.
Counting edges between $V(G)\setminus B$ and $B$ and adding $\binom{r+2}{2}$ edges inside $B$ gives
\[
e(G)=r(n-r-2)+\binom{r+2}{2}=rn+O(1).
\]
\q

\section{Hypergraph visualization of Theorem~\ref{T15}}\label{app:T15-matrices}

This appendix records the incidence-matrix form of the $(r-2)$-intersecting hypergraphs arising from Constructions~\textbf{(I)}--\textbf{(IV)}.
These matrices are intended only as a visualization of the neighborhood patterns that occur in the graph constructions of Appendix~\ref{app:T15-constructions}; they are \emph{not} used to prove Proposition~\ref{prop:T15-upper}.
For each $c$-interval in that proposition, we display a feasible solution to~\eqref{eq:1.3} whose objective value equals $f(r,c)$.

Following the idea of F\"uredi~\cite{Fur1}, we encode $\mathcal{H}$ by an $|\mathcal{E}|\times |V|$ \emph{incidence matrix}: the $(i,j)$-entry is $1$ if $b_j\in E_i$ and $\ind$ otherwise, where $V=\{b_1\ldots,b_{|V|}\}$ and $\mathcal{E}=\{E_1,\ldots,E_{|\mathcal{E}|}\}$.
A column of weights $(y_1,\ldots,y_{|\mathcal{E}|})^{\mathsf T}$ placed to the left of the matrix specifies a feasible solution to~\eqref{eq:1.3}.
The rows below are the distinct sets $N(x)\cap V$ for any $x\in V$ arising in cases~\textbf{(I)}--\textbf{(IV)} of Proposition~\ref{prop:T15-upper}.

\medskip
\noindent\textbf{Construction (I).} Let $\dfrac{r-1}{r} < c < 1$ and $\mathcal{H}$ be an $(r-2)$-intersecting hypergraph with $V=\{b_1,\ldots,b_r\}$ and $\mathcal{E}=\{E_1,\ldots,E_r\}$, where $E_i=V\setminus\{b_i\}$ for $1\le i\le r$. For each $1\le i\le r$, the weight of $E_i$ is $y_i=1/r$. The incidence matrix of $\mathcal{H}$:
columns are $b_1,\ldots,b_r$; row $E_i$  with weight $y_i=1/r$.
Take $r=6$ as an example:
\[
\begin{array}{r@{\hspace{0.75em}}c@{\hspace{0.5em}}l}
\dfrac{1}{r} &
\begin{pmatrix}
\ind&1&1&1&1&1\\
1&\ind&1&1&1&1\\
1&1&\ind&1&1&1\\
1&1&1&\ind&1&1\\
1&1&1&1&\ind&1\\
1&1&1&1&1&\ind
\end{pmatrix}.
\end{array}
\]Obviously, $a(\mathcal{H},c)= r-1$.

\medskip
\noindent\textbf{Construction (II).} Let $\dfrac{r{-}3/2}{r{-}1/2} < c < \dfrac{r{-}1}{r}$ and $\mathcal{H}$ be an $(r-2)$-intersecting hypergraph with $V=\{b_1,\ldots,b_r,b_{r+1}\}$ and $\mathcal{E}=\{E_1,E_2,\ldots,E_{r-1},E_{r}, E_{r+1}\}$, where $E_1=\{b_1,\ldots,b_{r-1}\}$ with weight $y_1= (r{-}1)c{-}(r{-}2)$,  $E_j=\{b_1,\ldots,b_r,b_{r+1}\}\setminus\{b_j\}$ with weight $y_j=1-c$ for $2\le j\le r-1$, $E_{r}=\{b_2,\ldots,b_{r-1},b_r\}$ and $E_{r+1}=\{b_2,\ldots,b_{r-1},b_{r+1}\}$ with weight $y_{r}=y_{r+1}=\tfrac{1-c}{2}$.
The incidence matrix of $\mathcal{H}$: columns are $b_1,\ldots,b_{r+1}$, row $E_i$  with weight $y_i$ for $1\le i\le r+1$.
Take $r=6$ as an example:
\[
\begin{array}{r@{\hspace{0.75em}}c@{\hspace{0.5em}}l}
\begin{array}{@{}c@{}} (r{-}1)c{-}(r{-}2) \\ 1{-}c \\ 1{-}c \\ 1{-}c \\ 1{-}c \\[0.2em]
(1-c)/2 \\[0.2em]
(1-c)/2 \end{array} &
\begin{pmatrix}
1&1&1&1&1&\ind&\ind\\
1&\ind&1&1&1&1&1\\
1&1&\ind&1&1&1&1\\
1&1&1&\ind&1&1&1\\
1&1&1&1&\ind&1&1\\
\ind&1&1&1&1&1&\ind\\
\ind&1&1&1&1&\ind&1
\end{pmatrix}.
\end{array}
\]
One can check that $\sum_{i=1}^{r+1} y_i=1$, $\sum_{i:b_j\in E_i}y_i\le c$ for every $1\le j\le r+1$, and
\[
	a(\mathcal{H},c)=
\sum_{i=1}^{r+1} |E_i|\,y_i
=(r{-}1)y_0+r(r{-}2)(1{-}c)+(r{-}1)(y_{r}{+}y_{r+1})
=2r-3-(r-2)c.
\]

\medskip
\noindent\textbf{Construction (III).} Let $\dfrac{r{-}2}{r{-}1} < c < \dfrac{r{-}3/2}{r{-}1/2}$ and $\mathcal{H}$ be an $(r-2)$-intersecting hypergraph with $V=\{b_1,\ldots,b_r,b_{r+2}\}$ and $\mathcal{E}=\{E_1,E_2,\ldots,E_{r-1},E_{r}, E_{r+1}\}$, where $E_1=\{b_1,\ldots,b_{r-1}\}$ with weight $y_1= (r{-}1)c{-}(r{-}2)$, $E_2=\{b_1,b_3,\ldots,b_r,b_{r+1}\}$ with weight $y_2=1{-}c$,
 $E_j=\{b_1,\ldots,b_{r-1},b_r,b_{r+2}\}\setminus\{b_j\}$ with weight $y_j=1-c$ for $3\le j\le r-1$, $E_{r}=\{b_2,\ldots,b_{r-1},b_r\}$ with weight $y_{r}=c{-}(r{-}2)(1{-}c)$ and $E_{r+1}=\{b_2,\ldots,b_{r-1},b_{r+1},b_{r+2}\}$ with weight $y_{r+1}=(r{-}1)(1{-}c){-}c$.
The incidence matrix of $\mathcal{H}$: columns are $b_1,\ldots,b_{r+2}$, row $E_i$  with weight $y_i$ for $1\le i\le r+1$.
Take $r=6$ as an example:
\[
\begin{array}{r@{\hspace{0.75em}}c@{\hspace{0.5em}}l}
\begin{array}{@{}c@{}} (r{-}1)c{-}(r{-}2) \\ 1{-}c \\ 1{-}c \\ 1{-}c \\ 1{-}c \\[0.2em]
c{-}(r{-}2)(1{-}c) \\[0.2em]
(r{-}1)(1{-}c){-}c \end{array} &
\begin{pmatrix}
1&1&1&1&1&\ind&\ind&\ind\\
1&\ind&1&1&1&1&1&\ind\\
1&1&\ind&1&1&1&\ind&1\\
1&1&1&\ind&1&1&\ind&1\\
1&1&1&1&\ind&1&\ind&1\\
\ind&1&1&1&1&1&\ind&\ind\\
\ind&1&1&1&1&\ind&1&1
\end{pmatrix}.
\end{array}
\]
One can check that $\sum_{i=1}^{r+1} y_i=1$, $\sum_{i:b_j\in E_i}y_i\le c$ for every $1\le j\le r+2$, and
\[
\begin{aligned}
	a(\mathcal{H},c)=
\sum_{i=1}^{r+1} |E_i|\,y_i
&=(r{-}1)y_0+(r{-}1)(1{-}c)+r(r{-}3)(1{-}c)+(r{-}1)y_{r}+r\,y_{r+1}\\
&=(r{-}1)+(2r{-}3)(1{-}c)-c =3r{-}4{-}(2r{-}2)c.
\end{aligned}
\]

\medskip
\noindent\textbf{Construction (IV).} Let $\dfrac{r}{r+2} < c < \dfrac{r-2}{r-1}$ and $\mathcal{H}$ be an $(r-2)$-intersecting hypergraph with $V=\{b_1,\ldots,b_r,b_{r+2}\}$ and $\mathcal{E}=\{E_1,E_2,\ldots, E_{r+2}\}$, where $E_i=V\setminus\{b_i,b_{i+1}\}$ (indices mod $r+2$) with weight $y_i=1/(r+2)$ for $1\le i\le r+2$.
The incidence matrix of $\mathcal{H}$: columns are $b_1,\ldots,b_{r+2}$, row $E_i$  with weight $y_i$ for $1\le i\le r+2$.
Take $r=6$ as an example:
\[
\begin{array}{r@{\hspace{0.75em}}c@{\hspace{0.5em}}l}
\dfrac{1}{r+2} &
\begin{pmatrix}
\ind&\ind&1&1&1&1&1&1\\
1&\ind&\ind&1&1&1&1&1\\
1&1&\ind&\ind&1&1&1&1\\
1&1&1&\ind&\ind&1&1&1\\
1&1&1&1&\ind&\ind&1&1\\
1&1&1&1&1&\ind&\ind&1\\
1&1&1&1&1&1&\ind&\ind\\
\ind&1&1&1&1&1&1&\ind
\end{pmatrix}.
\end{array}
\]
Obviously, $a(\mathcal{H},c)=r$.

Thus, on each $c$-interval in Proposition~\ref{prop:T15-upper}, the program~\eqref{eq:1.3} admits a feasible solution whose objective value equals~$f(r,c)=A_{r-2}(c)$.

\end{document}